\documentclass[11pt]{amsart}

\usepackage{amssymb,amsmath,amsthm}
\usepackage{stmaryrd}
\usepackage[all]{xy}
\usepackage{url}

\usepackage[
            pdftex,
            colorlinks=true,
            urlcolor=blue,       
            filecolor=blue,     
            linkcolor=blue,       
            citecolor=blue,         
            pdftitle= {Title},
            pdfauthor={Author},
            pdfsubject={Subject},
            pdfkeywords={Key}
            pagebackref,
            pdfpagemode=UseNone,
            bookmarksopen=true]
            {hyperref}
\usepackage{color}

 \newtheorem{thm}{Theorem}[section]
 \newtheorem{prop}[thm]{Proposition}
 \newtheorem{lem}[thm]{Lemma}
 \newtheorem{corollary}[thm]{Corollary}

 \theoremstyle{definition}
 \newtheorem{example}[thm]{Example}
 
 \newtheorem{remark}[thm]{Remark}

\numberwithin{equation}{section}

\newcommand{\bbZ}{{\mathbb{Z}}}

\newcommand{\bbP}{{\mathbb{P}}}

\newcommand{\bbC}{{\mathbb{C}}}

\newcommand{\bbk}{{\Bbbk}}

\newcommand{\tr}{\operatorname{Tr}}
\newcommand{\GL}{\operatorname{GL}}

\newcommand{\PGL}{\operatorname{PGL}}

\newcommand{\Aut}{\operatorname{Aut}}

\newcommand{\Proj}{\operatorname{Proj}}
\newcommand{\rank}{\operatorname{rank}}
\newcommand{\corank}{\operatorname{corank}}

\newcommand{\Or}{\operatorname{Or}}

\newcommand{\Bs}{\operatorname{Bs}}

\newcommand{\Idem}{\operatorname{Idem}}

\newcommand{\Pic}{\operatorname{Pic}}
\newcommand{\Sing}{\operatorname{Sing}}

\newcommand{\Pf}{\operatorname{Pf}}
\newcommand{\rad}{\operatorname{rad}}
\newcommand{\hdisc}{\tfrac{1}{2}\operatorname{disc}}
\newcommand{\disc}{\operatorname{disc}}

\newcommand{\cha}{\operatorname{char}}

\newcommand{\Arf}{\operatorname{Arf}}

\newcommand{\bsm}{\left(\begin{smallmatrix}}
\newcommand{\esm}{\end{smallmatrix}\right)}

\newcommand{\la}{\langle}
\newcommand{\ra}{\rangle}

\newcommand{\calB}{\mathcal{B}}

\newcommand{\calE}{\mathcal{E}}
\newcommand{\calF}{\mathcal{F}}

\newcommand{\calI}{\mathcal{I}}

\newcommand{\calO}{\mathcal{O}}

\newcommand{\frakS}{\mathfrak{S}}

\newcommand{\sfq}{{\mathsf{q}}}
\newcommand{\sfb}{{\mathsf{b}}}

\newcommand{\sfP}{{\mathsf{P}}}

\newcommand{\sfD}{{\mathsf{D}}}

\newcommand{\beq}{\begin{equation}}
\newcommand{\eeq}{\end{equation}}

\begin{document}

\title[Regular pairs of quadratic forms]{Regular pairs of quadratic forms on
odd-dimensional spaces in characteristic 2}
\author{Igor Dolgachev}
\author{Alexander Duncan}
\thanks{The second author was partially supported by
National Science Foundation
grant DMS 0943832 and National Security Agency grant H98230-16-1-0309.}

\begin{abstract}
We describe a normal form for a smooth intersection of two quadrics
in even-dimensional projective space over an arbitrary
field of characteristic 2.
We use this to obtain a description of the automorphism group of such
a variety.
As an application, we show that every quartic del Pezzo surface over a
perfect field of characteristic 2 has a canonical rational point
and, thus, is unirational. 
\end{abstract}

\maketitle

\section{Introduction}

Let $(q_0,q_1)$ be a pair of quadratic forms on a 
vector space $E$ over a field $\bbk$.
The common zeros of the pair define a subvariety $X = V(q_0,q_1)$ in the
projective space  $|E|=\bbP(E^\vee)$ of lines in $E$.
We say the pair is \emph{regular} if the variety $X$
is smooth and of codimension $2$ in $|E|$.
One may also consider the pencil of quadrics spanned by
$V(q_0)$ and $V(q_1)$; the pencil is \emph{regular} if $X$ is  smooth.

When $\Bbbk = \bbC$, the field of complex numbers, a classical result
due to A. Cauchy and C. Jacobi (see \cite{Muth}) states that any
regular pair can be \emph{simultaneously diagonalized}.
When $q_0$ is non-degenerate,
this means that we may find a basis $x_1,\ldots,x_n$ in the dual space
$E^\vee$ such that
\[
q_0 = \sum_{i=1}^n x_i^2,
\quad\quad q_1 = \sum_{i=1}^n a_ix_i^2\ .
\]
Here the coefficients $a_1,\ldots, a_n$ are distinct roots of the
discriminant polynomial $\delta(t) = \det(t M_0-M_1)$, 
where $M_0$ and $M_1$ are symmetric matrices representing the polar
symmetric bilinear forms associated to $q_0$ and $q_1$.
When the pair is not necessarily regular, the classification was carried
out by K. Weierstrass~\cite{Weir} by considering the elementary divisors
of $t M_0-M_1$.
Over an arbitrary field $\bbk$ of characteristic not $2$, pairs of
quadratic forms were classified by L.~Kronecker and L.~Dickson.
A modern exposition of their theory can be found
in~\cite{Wat76} (see also~\cite{LS99}).

Unsurprisingly, in characteristic $p = 2$ the situation is more
complicated.
The main reason for this is that one can no longer identify quadratic forms
with symmetric bilinear forms.
Even worse, when $n=\dim(E)$ is odd the determinant of
$t M_0-M_1$ is identically zero.
One may still consider symmetric and alternating bilinear forms
in this case
(see, for example, \cite{Ish}, \cite{Wat77}, \cite{LS99}),
but the connection with quadratic forms is more tenuous.

The geometry of the intersection $X$ of two quadrics differs drastically
depending on whether $n$ is even or odd.
Examples when $n$ is even include quartic elliptic curves in $\bbP^3$ and
quadratic line complexes in $\bbP^5$.
When $n$ is odd, the first interesting example is
a quartic del Pezzo surface in $\bbP^4$ isomorphic to 
a blow-up of 5 points in the projective plane.

When $n=2m$ is even and $\Bbbk = \bbC$, one can associate to $X$ its
intermediate Jacobian $J(X)$.
A theorem of Luc Gauthier and Andr\'e Weil \cite{Gau}
asserts that $J(X)$ is isomorphic to the Jacobian
variety of a hyperelliptic curve $C$ of genus $m-1$ with
equation $y^2+\delta(t) = 0$.
Over an algebraically closed field of characteristic $p\ne 2$,
the Jacobian is isomorphic to the variety of $(m-2)$-planes in $X$
(see also \cite{Wang}).
If $p = 2$ and $n$ is even, one should use the Pfaffian $P(t)$ defined
so that $P(t)^2 = \delta(t)$.
Under the condition that the roots $\alpha_1, \ldots, \alpha_m$
of $P(t)$ are distinct, U.~Bhosle~\cite{Bhosle} provides a normal form
\[
q_0 = \sum_{i=1}^m x_ix_{m+i} \ , \quad \quad
q_1 = \sum_{i=1}^m
\left( a_i x_i^2 + \alpha_i x_ix_{m+i} + b_i x_{m+i}^2\right)
\]
where $a_1,\ldots,a_m$ and $b_1,\ldots,b_m$ are in $\bbk$.
The variety of $(m-2)$-planes is isomorphic to the
Jacobian of a hyperelliptic curve in this case as well.

This paper is concerned with pairs of quadratic forms, and their
corresponding pencils, in the case where $\bbk$ has characteristic $2$
and $n=\dim(E)=2m+1$ is odd.
As mentioned above, the discriminant polynomial is identically zero
in this case.
Instead, we use the \emph{half-discriminant} introduced by
M.~Kneser~\cite{Kneser}, which we recall in Section~\ref{sec:prelims}.
This is a homogeneous polynomial of degree $n$
\begin{equation} \label{eq:halfDisc}
\Delta(t_0,t_1) : = \hdisc(t_0q_0+t_1q_1) =  a_0 t_0^n +
a_1 t_0^{n-1}t_1 + \ldots + a_n t_1^n \end{equation} in
$\bbk[t_0,t_1]$.
This polynomial behaves like the usual discriminant when the
characteristic is not $2$:
the pencil is regular if and only if the zero subscheme $V(\Delta)$ in
$\bbP^1$ is smooth of dimension $0$
(we give a geometric proof, an algebraic proof can be found 
in \cite{LS02}).
When $\bbk$ is algebraically closed, this simply means that
$\Delta \not\equiv 0$ and $V(\Delta)$ consists of $n$ distinct points.

Our first result is the following:

\begin{thm}
\label{thm:NormalForm}
Let $(q_0,q_1)$ be a regular pair of quadratic forms on a vector
space $E$ of dimension $n=2m+1 \ge 3$ over a field $\bbk$ of
characteristic 2. Then there exists a basis
$(x_0,\ldots,x_{m},y_0,\ldots,y_{m-1})$
in $E^\vee$ such that 
\begin{align}\label{eq:main}
\begin{split}
q_0& = \sum_{i=0}^{m}a_{2i}x_i^2+
\sum_{i=0}^{m-1} x_{i+1}y_i+\sum_{i=0}^{m-1}r_{2i+1}y_i^2,\\
q_1& =
\sum_{i=0}^{m}a_{2i+1}x_i^2+\sum_{i=0}^{m-1}
x_{i}y_i+\sum_{i=0}^{m-1}r_{2i}y_i^2,
\end{split}
\end{align}
where the coefficients $a_0,\ldots, a_{n}$ are equal to those of the
half-discriminant polynomial \eqref{eq:halfDisc},
and $r_0,\ldots,r_{n-2}$ are in $\bbk$.
\end{thm}

To the authors' surprise, our normal form seems to be new even in the
case $n = 5$ corresponding to quartic del Pezzo surfaces. 
At least in retrospect, it can be easily deduced from the 
classification of pairs of alternate bilinear forms
(see \cite{LS99}).
Nevertheless, the existence of the normal form has the following
arithmetic consequence particular to the case of characteristic $2$.

\begin{corollary} \label{cor:Xrational}
Let $X$ be a smooth complete intersection of two quadrics in $\bbP^{2m}$
defined over a perfect field $\bbk$ of characteristic $2$.
If $\dim(X) = 2$ (resp. $> 2$), then $X$ is unirational (resp. rational) over 
$\Bbbk$.
\end{corollary}

(Note that the unirationality of del Pezzo surfaces of degree $4$
is already known for finite fields, see \cite{MT}.)

Following \cite{Reid}, a \emph{generator} is a linear subspace of $X$
of dimension $m$.
Over the algebraic closure, there are exactly $2^{2m}$ generators.
Following \cite{Skor}, one says $X$ is \emph{quasi-split} if it
contains a generator, and $X$ is \emph{split} if all $2^{2m}$ generators
are defined over the base field.

\begin{thm} \label{thm:geometricRis0}
A regular pair $(q_0,q_1)$ has quasi-split base locus $X$
if and only if the pair has a normal form such that
$r_0=\cdots=r_{n-2}=0$.
In particular, this is always possible when $\bbk$ is algebraically
closed.
\end{thm}

If $a_n \ne 0$, we assign to the coefficients $(r_0, \ldots, r_{n-2})$
an element $r$ of the $\bbk$-algebra $A=\bbk[T]/(f(T))$ where $f(T)$ is the
dehomogenization of $\Delta$ with respect to $T=t_1/t_0$.
This element, which we call the \emph{$r$-invariant}, determines the
normal form (it can also be defined when $a_n = 0$, see
Remark~\ref{rem:technicalNuisance}).
An \emph{isomorphism} between pairs $(q_0,q_1)$, $(q'_0,q'_1)$ of
quadratic forms is an element $g \in \GL(E)$ such that
\[
q'_0(v) = q_0(gv) \textrm{ and } q'_1(v) = q_1(gv)
\]
for all $v \in E$.

\begin{thm} \label{thm:OtherNormalForms}
Suppose two regular pairs $(q_0,q_1)$ and $(q_0',q_1')$
have the same half-discriminant polynomial and
have normal forms with invariants $r$ and $r'$, respectively.
The pairs are isomorphic if and only if 
\[ r \equiv r' \mod \bbk + \wp(A) \]
where $\wp(s) = s^2+s$ is the Artin-Schreier map $A\to A$.
\end{thm}

Denote by $\Aut(q_0,q_1)$ be the group of automorphisms of the pair
$(q_0,q_1)$.
When $X$ is split,
the group of automorphisms $\Aut(q_0,q_1)$ of a regular pair $(q_0,q_1)$
is isomorphic to an elementary abelian $2$-group of rank $2m$,
which acts simply transitively on the set of generators of $X$.
In general, the set of generators can be viewed as an
$\Aut(q_0,q_1)$-torsor.

\begin{thm} \label{thm:GaloisCohomology}
Let $\wp(s) = s^2+s$ be the Artin-Schreier map $A\to A$.
There is an isomorphism between
the algebra $A/(\wp(A)+\bbk)$ and
the Galois cohomology group $H^1(\bbk,\Aut(q_0,q_1))$,
which takes the $r$-invariant to the
class of the $\Aut(q_0,q_1)$-torsor of generators.
\end{thm}

This theorem provides a Galois-cohomological interpretation of
Theorems~\ref{thm:geometricRis0} and \ref{thm:OtherNormalForms}
in the vein of \cite{Skor}.
Namely, the isomorphism class of a regular pair is determined by its
half-discriminant polynomial $\Delta$ and its torsor of generators.
Similarly, the base locus $X$ is determined up to isomorphism by
a smooth subscheme $V(\Delta)$ of $\bbP^1$
and the torsor of generators;
the variety $X$ is quasi-split if and only if the torsor is trivial.

One of the more concrete motivations of this paper is to study the
possible groups of automorphisms of a quartic del Pezzo surface over an
algebraically closed field of characteristic 2
(in order to study the conjugacy classes of finite
subgroups of the plane Cremona group over fields of positive
characteristic).

Recall that a \emph{reflection} is an involution of a vector space (or
a projective space) which leaves pointwise-fixed a hyperplane.
As a corollary of our main results about the classification of pairs of
quadratic forms we obtain the following.

\begin{thm} \label{thm:autos}
Let $X=V(q_0,q_1)$ be a smooth complete intersection of two quadrics
in $\bbP^{2m}$ over an algebraically closed field $\bbk$ of
characteristic $2$.
Then:
\[ \Aut(X) = \Aut(q_0,q_1) \rtimes G \]
where $\Aut(q_0,q_1)$ is generated by reflections
in canonical bijection with the points of $V(\Delta)$
and $G$ is isomorphic to the subgroup of $\Aut(\bbP^1)$ which leaves
invariant the points $V(\Delta)$.
\end{thm}

This extends to characteristic 2 a classical result on automorphisms of
quartic del Pezzo surfaces (see \cite{CAG}).
There is a natural action of the Weyl group
$W(D_n) \cong 2^{2m} \rtimes \frakS_n$
on the cohomology classes of the generators.
The automorphism group $\Aut(X)$ is naturally a subgroup of
the Weyl group $W(D_n)$ where $\Aut(q_0,q_1) \cong 2^{2m}$
and $G \subset \frakS_n$.
Thus, as with quartic del Pezzo surfaces,
the automorphism group is completely determined by its action on the
generators.

In writing this paper, the authors tried to resolve a tension between
two likely audiences: geometers who only work over
algebraically closed fields, and algebraists without a strong background
in geometry.
Hopefully, the paper will appeal to both audiences
and we occasionally supply multiple proofs to facilitate this.
For example, while Theorem~\ref{thm:geometricRis0} is an immediate
consequence of Theorem~\ref{thm:OtherNormalForms},
we provide a more direct geometric proof in Section~\ref{sec:quasiSplit}.
Also, while Theorem~\ref{thm:OtherNormalForms} is used in the proof
of Theorem~\ref{thm:autos}, a geometric description of the reflection
group $R$ can be found in Section~\ref{sec:geomAutos}.

The remainder of the paper is structured as follows.
In Sections \ref{sec:prelims} and \ref{sec:pencils} we establish some
preliminaries on pencils of quadratic forms in characteristic $2$.
In Section~\ref{sec:normalForms}, we prove Theorem~\ref{thm:NormalForm},
Corollary~\ref{cor:Xrational}, and Theorem~\ref{thm:geometricRis0}.
In Section~\ref{sec:NormalFormIsos}, we define the $r$-invariant and prove
Theorem~\ref{thm:OtherNormalForms}.
In Section~\ref{sec:arf}, we discuss possible interpretations for what
the $r$-invariant represents.
In Section~\ref{sec:autos}, we determine the automorphism group of a
pair of quadratic forms and use this to prove
Theorems~\ref{thm:GaloisCohomology}~and~\ref{thm:autos}.
Finally, in Section~\ref{sec:cohomology}, we make some observations
about the cohomology of smooth intersections of smooth quadrics
extending results from  M.~Reid's thesis \cite{Reid} to the case where
the characteristic is $2$.

\section{Preliminaries on quadratic forms}
\label{sec:prelims}

Throughout $\bbk$ is a field of characteristic $2$
and $E$ is a $k$-vector space of dimension $n$.

\subsection{Alternating bilinear forms}

Recall that a bilinear form $b : E \times E \to \bbk$ is
\emph{alternating} if $b(v,v)=0$ for all $v \in E$.
We may view $b$ as an element in $\bigwedge^2 E^\vee$.
The \emph{radical} $\rad(b)$ of an alternating bilinear form $b$ is
the kernel of the induced map $E \to E^\vee$.
A form $b$ is \emph{nondegenerate} if its radical is trivial.
The \emph{corank} of a bilinear form is $\dim(\rad(b))$.
Every alternating bilinear form is an orthogonal sum of its radical and
a nondegenerate alternating bilinear form.

Every nondegenerate alternating bilinear form has even dimension $n=2m$
and has a \emph{symplectic basis} $v_1,\ldots,v_m,w_1,\ldots,w_m$
satisfying the relations
\[
\begin{array}{rcl c rcl c rcl}
b(w_i,w_j)&=&0, & &
b(v_i,v_j)&=&0, & &
b(w_i,v_j)&=&\delta_{ij},
\end{array}
\]
for $1 \le i,j \le m$.
A subspace $W$ of $b$ is \emph{totally isotropic} if
$b(v,w)=0$ for all $v,w \in W$.
Note that $v_1,\ldots,v_m$ and $w_1,\ldots,w_m$ are bases for
complementary totally isotropic subspaces of $E$.
Conversely,
for any pair of complementary totally isotropic subspaces of $E$,
there exists a basis for each subspace such that the union is a
symplectic basis for $E$.

When $n$ is even (resp. odd), $\corank(b)$ is even (resp. odd).
Thus, an alternating bilinear form $b$ on a vector space of odd
dimension $n=2m+1$ is always degenerate.

\begin{prop} \label{prop:Omega}
Let $b$ be an alternating bilinear form 
on a vector space $E$ of odd dimension $n=2m+1$.
Up to a choice of volume form for $E$,
there is a canonical vector $\omega \in E$
which spans $\rad(b)$ if $\corank(b)=1$ and is $0$ otherwise.
Choosing a basis of $e_1,\ldots,e_n$ in $E$ such that
$e_1 \wedge \cdots \wedge e_n \mapsto 1$ under the volume form,
we have 
\begin{equation}\label{eq:pfafcoord}
\omega = (\Pf_1,\ldots,\Pf_n),
\end{equation}
where $\Pf_i$ is the Pfaffian of the principal submatrix of the matrix
of $b$ obtained by deleting the $i$th row and the $i$th column.
\end{prop}

\begin{proof}
It is obvious that such a vector exists, the point is to construct it
canonically and find an explicit formula for its coordinates.

Let us view $b \in \bigwedge^2E^\vee$ as an element of the divided power
algebra of $E^\vee$.
We may thus consider the $m$th divided power $b^{(m)} \in
\bigwedge^{2m}E^\vee$.
(If $\bbk$ was of characteristic zero, we would have the formula
$b^{(m)} = (b \wedge \cdots \wedge b)/m!$.)
Under the isomorphism $\bigwedge^{2m}E^\vee \cong E$
defined by the volume form, $b^{(m)}$ maps to
a vector $\omega$ in $E$. We have 
$\omega = \sum_{i=1}^n\Pf_ie_i$ in coordinates.
The rank of an alternating matrix is equal to the size of the
largest non-zero Pfaffian of a principal submatrix of even size.
If the rank is equal to $n-1$, then the
nullspace is generated by the vector \eqref{eq:pfafcoord}.
All of these well-known facts can be found in \cite{Buchsbaum}.
\end{proof}

Note that without a choice of volume form, $\omega$ is well-defined
up to a multiplicative constant.
Thus, the corresponding point in the projective space $|E|$
does not depend on this choice.

\subsection{Quadratic forms}

Recall that a quadratic form $q$ is an element of the symmetric square
$S^2(E^\vee)$ of the dual space $E^\vee$.
When $q\ne 0$, one obtains a quadric hypersurface $V(q)$ in
the projective space $|E|$.
The quadratic form is called \emph{nondegenerate} if $V(q)$ is a smooth
hypersurface.

Equivalently, $q$ can be viewed as a function $E \to \bbk$ such that
\begin{enumerate}
\item[(1)] $q(cv)=c^2q(v)$ for all $c \in \bbk, v \in E$, and
\item[(2)] $b(v,w)=q(v+w)-q(v)-q(w)$ is a symmetric bilinear form.
\end{enumerate}
The bilinear form $b$ is called the
\emph{associated polar bilinear form}
and, since $\cha(k)=2$, we observe that $b$ is alternating. 
The \emph{corank} of a quadratic form $q$ is simply the corank of the
associated bilinear form.

A non-zero vector $v \in E$ is called a \emph{singular vector of $q$}
if $q(v)=0$ and $v \in \rad(b)$.
In geometric language, this means the corresponding point $[v]=\bbk v$
in $|E|$ is a singular point of the quadric $V(q)$.

A quadratic form is \emph{totally singular} if its associated
bilinear form is trivial;
this is equivalent to $V(q)$ being singular at every
geometric point.
A quadratic form $q$ is diagonalizable if and only if it is totally
singular.
Moreover, if $q$ is diagonalizable, then it is diagonal relative to any basis.
A subspace $W$ of $E$ is \emph{totally isotropic}
(resp. \emph{totally singular}) with respect to $q$ if
the restricted form $q|_W$ is trivial (resp. totally singular).

The \emph{discriminant} $\disc(q)$ of a quadratic form is the
determinant of the matrix of the polar bilinear form $b$.
If $n$ is even, then a quadratic form is nondegenerate if and
only if $\disc(q) \ne 0$.
When $n$ is odd, $\disc(q)$ is always zero since the polar bilinear form
has odd corank; thus, we consider another invariant.

We define the \emph{half-discriminant} of $q$ as
\[
\hdisc(q): = q(\omega) 
\]
where $\omega$ is the canonical vector from Propositon~\ref{prop:Omega}.
In coordinates, if
\[
q = \sum_{1\le i\le j\le n}a_{ij}x_ix_j
\]
for a basis $x_1,\ldots,x_n$ for $E^\vee$, then
\begin{equation}\label{eq:discExpl}
\hdisc(q) = \sum_{1\le i\le j\le n}a_{ij}\Pf_i\Pf_j \ .
\end{equation}
Since each expression $\Pf_i$ is a homogeneous polynomial of degree $m$
in the coefficients $\{ a_{ij} \}$,
the half-discriminant is either zero or a homogeneous polynomial of degree
$n = 2m+1$ in the coefficients $\{ a_{ij} \}$. 
Recall that $\omega$ is only defined up to a multiplicative constant which
depends on the choice of volume form.
If we do not make such a choice, then $\hdisc(q)$ is only well-defined
modulo non-zero squares in $\bbk$.

\begin{prop} \label{prop:halfDet}
For $n$ odd, $\hdisc(q) \ne 0$ if and only if $V(q)$ is smooth.
\end{prop}

\begin{proof}
If $\dim~\rad(b) > 1$, then the associated projective subspace
intersects the quadric $Q = V(q)$;
in this case $Q$ is not smooth, $\omega=0$, and $\hdisc(q)=0$.
Otherwise, $\omega$ spans $\rad(b)$
and $Q$ is singular if and only if $q(\omega)=0$.
\end{proof}

\begin{remark}
The half-discriminant was first by introduced by M.~Kneser.
Our definition is equivalent to that of \cite[p. 43]{Kneser}.
Indeed, both polynomials define the same reduced hypersurface in the
projective space of the coordinates $a_{ij}$ (they vanish precisely when
$V(q)$ is singular).
Their degrees are both equal to $n$, hence they differ by a constant factor.
It remains only to compare their values on a particular quadratic form
to see that they are equal.
\end{remark}

\begin{remark}
The half-discriminant is also known as
the \emph{corrected discriminant} (see Appendix~17.6~of~\cite{EGA}),
the \emph{half-determinant} (see \cite{LS02}),
or the \emph{determinant} (see Remark~13.8~of~\cite{EKM}).
\end{remark}

\begin{example} If $n  = 3$, we have
\[
\hdisc(q) = a_{11}a_{23}^2+a_{22}a_{13}^2+a_{33}a_{12}^2+a_{12}a_{23}a_{13}.
\]
The locus of zeros $V(\hdisc(q))$ in $|S^2E^\vee|\cong \bbP^5$ is a cubic
hypersurface singular along the  plane of double lines $a_{12} = a_{13}
= a_{23} = 0$.

If $n = 5$, then
\[
\hdisc(q) =
(a_{11}a_{23}^2a_{45}^2+\cdots)+
(a_{12}^2a_{34}a_{45}a_{35}+\cdots)+
(a_{12}a_{23}a_{34}a_{45}a_{15}+\cdots)
\ .
\]
\end{example}

\section{Pencils of quadratic forms}
\label{sec:pencils}

For the rest of the paper, we assume that the $\bbk$-vector space $E$
has dimension $n=2m+1$ for a positive integer $m$.

Let $U$ be a $2$-dimensional $\bbk$-vector space.
A \emph{pencil of quadratic forms} is an injective linear map
\[
\sfq : U \to S^2E^\vee.
\]
We use subscript notation $\sfq_u$ to denote the image of $u \in U$ in
order to avoid confusion with the interpretation of a quadratic form as
a map $E \to \bbk$.
We will also identify $\sfq$ with its image considering a pencil as a
two-dimensional linear subspace of $S^2E^\vee$.
The pencil $\sfq$ defines a $1$-dimensional linear system of
quadric hypersurfaces in $|E|$ (also called a pencil).
The corresponding rational map $|E|\dasharrow |U^\vee|$ has base
subscheme $\Bs(\sfq)$ equal to the intersection of quadrics $V(\sfq_u)$
over all $u \in U$.

We may view the map $\sfq$ as an element of $U^\vee\otimes S^2E^\vee$. 
Selecting a basis $(u_0, u_1)$ for $U$ with the dual basis $(t_0,t_1)$,
we can view $\sfq$ as an element $t_0q_0+t_1q_1\in U^\vee\otimes
S^2E^\vee$, where 
\[
q_0 = \sfq_{u_0}, \quad q_1 = \sfq_{u_1}
\]
define a pair of quadratic forms $(q_0,q_1)$ which form a basis of the pencil.
We have $\Bs(\sfq) = V(q_0,q_1) = V(q_0)\cap V(q_1)$. 
Conversely, a pair of non-proportional quadratic forms gives rise to a
pencil of quadratic forms $\sfq$ (along with a choice of basis for $U$).

Given a pencil of quadratic forms we obtain an 
\emph{associated pencil of alternating bilinear forms}
\[
\sfb : U \to \bigwedge^2E^\vee
\]
by setting $\sfb_u$ as the associated bilinear form of $\sfq_u$ for each
$u \in U$.
Note that, unlike $\sfq$, the map $\sfb$ may not be injective.
We may view $\sfb$ as an element of $U^\vee\otimes \bigwedge^2E^\vee$.

\subsection{The radical map and the radical subspace}

Fix a volume form on $E$.
For each alternating bilinear form $\sfb_u$ there exists a canonical
vector $\omega$ as in Proposition \ref{prop:Omega}.
Thus, we obtain a function
\begin{equation}
\label{eq:OmegaFunction}
\Omega : U \to E
\end{equation}
which is defined by polynomials of degree $m$ in view of
\eqref{eq:pfafcoord}.
Thus, abusing notation, we may view $\Omega$ as an element of
$E \otimes S^mU^\vee$.
We call the map \eqref{eq:OmegaFunction} the \emph{radical map} of the pencil.
As $u$ varies over $U$, the $\Omega(u)$ span a subspace $W$ of $E$
which we call the \emph{radical subspace} of $E$.
Equivalently, $\Omega$ gives rise to a $\bbk$-linear map
\begin{equation}
(S^mU^\vee)^\vee \to E
\end{equation}
whose image equals the radical subspace $W$.
(Note that we should not identify $(S^mU^\vee)^\vee$ and $S^mU$ in
characteristic $2$).

We define a function $\Delta : U \to \bbk$ via the composition
\[ \Delta(u) = \sfq_u(\Omega(u)) \]
for $u \in U$.
We call $\Delta$ the \emph{half-discriminant of the pencil $\sfq$} since
$\Delta(u)$ is the half-discriminant of each $\sfq_u$.
Since $\sfq$ can be viewed as a linear combination of quadratic forms with
linear coefficients in $U^\vee$
and $\Omega$ is defined by polynomials of degree $m$ in $U^\vee$,
$\Delta$ is an element of $S^n(U^\vee)$.
After choosing a basis $(t_0,t_1)$ in $U^\vee$, it can be identified
with a homogeneous polynomial of degree $n = 2m+1$ in $t_0,t_1$. 

\subsection{Regular pencils}
We say a pencil (or pair) is \emph{regular} if the base locus
$X=\Bs(\sfq)$ is a smooth variety over $\bbk$.

\begin{remark}
We borrowed the term regular from the classical terminology of regular
linear systems of quadric hypersurfaces.  It should not be confused with
the regularity of the base scheme $\Bs(\sfq)$ since, over a non-perfect
field $\Bbbk$, the latter could be a regular scheme but not smooth. 
\end{remark}

From Theorem~5.1~of~\cite{LS02}, we know that a pencil is regular
if and only if $\Delta$ is non-zero and has $n$ distinct linear factors
over the algebraic closure.  In other words:

\begin{thm} \label{thm:smooth}
A pencil $\sfq$ is regular if and only if,
over an algebraic closure of $\bbk$,
there are exactly $n$ degenerate quadrics in $|\sfq|$.
\end{thm}

We will need the following consequence of this theorem:

\begin{corollary} \label{cor:corank}
If $\sfq$ is a regular pencil, then every
non-zero quadratic form in the pencil has corank $1$.
\end{corollary}

\begin{proof}
It suffices to assume $k$ is algebraically closed.
Suppose the corank of some non-zero $q$ in $\sfq$ is $> 1$.
Then the radical $\rad(b)$ of the associated bilinear form $b$ is of
dimension $\ge 3$.
This implies that $|\rad(b)|$ contains a plane intersecting
$X$ non-trivially.
This shows that $V(q)$ has a singular point $x$ in $X$.
The base locus $X$ is smooth,
so the tangent space of $X$ at $x$ is of codimension $2$.
We have $X = V(q) \cap V(q')$ for some $q' \in \sfq$.
Thus, the tangent spaces of $V(q)$ and $V(q')$ at $x$ intersect along
a subspace of codimension $2$.
This implies that $V(q)$ is smooth at $x$, a contradiction.
\end{proof}

\subsection{Alternate proof of Theorem~\ref{thm:smooth}}

In the remainder of this section, we present a geometric proof that
works in arbitrary characteristic which we believe is of independent
interest.  It is not necessary for the rest of the paper.
Note that Theorem~\ref{thm:smooth} is well known when the characteristic
of $\bbk$ is odd.
Unlike the situation for the rest of the paper, in this section
we assume $\bbk$ is algebraically closed of arbitrary characteristic
and that the integer $n$ is odd only if the characteristic is even.

Let $\sfD$ be the subvariety of singular quadrics in $S^2E^\vee$
(this is just defined by the discriminant or half-discriminant).
Let $\sfD_0$ be the open subvariety of $\sfD$ consisting of quadrics
which have a unique isolated singular point
(an open condition since it is determined by the rank of the
matrix of the associated bilinear form).

\begin{lem}\label{lem:reg}
Let $L$ be a linear subspace of $S^2E^\vee$ of dimension $r > 0$.
Then the tangent space of $\sfD \cap L$ at a point $q\in
\sfD_0 \cap L$ with singular vector $v_0 \in E$ is the
linear space $\{q\in L: q(v_0) = 0\}$.
\end{lem}

\begin{proof}
Let
\[
\tilde{\sfD} = \{(q,v)\in S^2E^\vee\setminus \{0\}\times
E\setminus \{0\}: [v]\in \Sing(V(q))\}.
\]
Consider the maps $p_1 : \tilde{\sfD} \to S^2E^\vee$ and
$p_2:\tilde{\sfD}\to E$ obtained from the projections.
For any $v\in E\setminus \{0\}$,
let $F_v$ be the closure of $p_1(p_2^{-1}(v))$ in $S^2E^\vee$.
By choosing a basis in $E$ such that $v = (1,0,\ldots,0)$,
the subvariety $F_v$ consists of quadratic forms that do not
contain the first variable.
Thus, $F_v$ is a linear subspace of $S^2E^\vee$ of codimension $n$.
In particular,
$\tilde{\sfD}$ is a fibration over $E\setminus \{0\}$ with
fibers isomorphic to linear spaces of the same dimension.
It follows that $\tilde{D}$ is smooth and of dimension
$n+\dim S^2E^\vee-n = \dim S^2E^\vee$.

The variety $p_2(p_1^{-1}(q))$ is, set-theoretically, the linear
space of singular vectors.
Passing to the map of the associated
projective varieties $\pi : |\tilde{\sfD}|\to |S^2E^\vee|$, we
expect that the map $\pi$ induces a birational isomorphism with the
subvariety $|\sfD_0|$ of quadrics with isolated singular point.

We will compute the differential of the map
$p_1:\tilde{\sfD}\to S^2E^\vee$ at a point $(q_0,v_0)$ represented by a
quadratic form $q_0$ with one-dimensional space of singular vectors
generated by $v_0$.

Let  $t_1,\ldots,t_{n}$ be coordinates  in $E$ and let $A_{ij}$  be
coordinates in  $S^2E^\vee$ corresponding to the coefficients of a
quadratic form.   
Then $\tilde{\sfD}$ is given  by $n+1$  equations:
\begin{eqnarray}
F_k &=& \sum_{j=1}^{n} B_{kj}t_j = 0, \quad k = 1,\ldots,n,\\
F_{n+1} &=&  \sum_{1\le i\le j\le n}A_{ij}t_it_j = 0,
\end{eqnarray}
where $B_{ii} = 2A_{ii}$ and $B_{ij} = B_{ji} = A_{ij}$ if $i < j$.

Let  $q_0 = \sum_{1\le i\le j\le n} a_{ij}t_it_j$ and $v_0 =
(c_1,\ldots,c_{n})$. The embedded tangent space $T_{q_0,v_0}(\tilde{\sfD})$ of
$\tilde{\sfD}$ is a subspace of
$T_{q_0}(S^2E^\vee)\oplus T_{v_0}(E) = S^2E^\vee\oplus E$
given by a system of linear equations in variables
$t_i,A_{ij}$ with matrix of coefficients $[M_1 M_2](a_{ij},c_i)$,
where 
\begin{eqnarray}
M_1 &=& \left(\frac{\partial F_k}{\partial t_j}\right)_{
1\le k\le n+1, 1\le j\le n}, \textrm{ and }\\ 
M_2 &=& \left(\frac{\partial F_k}{\partial A_{ij}}\right)_{
1\le k\le n+1, 1\le i\le j\le n}.
\end{eqnarray}
Computing the partial derivatives, we find that 
$T_{q_0,v_0}(\tilde{\sfD})$ consists of pairs $(q,v)$ satisfying the conditions
\begin{equation}\label{eq:tangent}
b_0(v,w) + b(v_0,w) = 0, \quad b_0(v_0,v)+q(v_0) = q(v_0) = 0,
\end{equation}
for all $w \in W$,
where $b$ (respectively $b_0$) is the associated bilinear form of
$q$ (respectively $q_0$).

The kernel of the differential of the map
$\tilde{\sfD} \to \sfD, (q,v) \mapsto q$
obtained from the projection 
can be identified with the linear space of vectors $v\in E$
such that $b_0(v,w) = 0$ for all $w \in E$;
this is the radical of the bilinear form $b_0$.
The map
\[T_{(q_0,v_0)}\tilde{\sfD}\to \{q:q(v_0)=0\}, (q,v)\to q\]
is linear and has kernel $\rad(b_0)$.
Since the tangent space of $\sfD_0$ at
$q_0$ is isomorphic to the quotient space
$T_{(q_0,v_0)}\tilde{\sfD}/\rad(b_0)$, we obtain that
\[ T_{q_0}\sfD =  \{q\in S^2E^\vee:q(v_0) = 0\}. \]

Note that $\rad(b_0)$ is one-dimensional
(this does not hold when $n$ is even and $\bbk$ is
characteristic $2$).
This is equal to the dimension of the fiber $p_1^{-1}(q_0)$.
This shows that the differential of $p_1$ is of maximal
rank, and hence $\pi$ is an isomorphism over $|\sfD_0|$.
Now let $L\subset S^2E^\vee$ be a linear non-zero subspace of
$S^2E^\vee$, then the variety of quadratic forms in $L$ with
one-dimensional space of singular vectors is equal to $L\cap \sfD_0$.
We have
\[ T_{q_0}(L\cap \sfD) = T_{q_0}L\cap T_{q_0}\sfD =
L\cap T_{q_0}\sfD = \{q\in L:q(v_0) = 0\}.\]
This proves the assertion. 
\end{proof}

\begin{proof}[Proof of Theorem~\ref{thm:smooth}]
We may assume $\bbk$ is algebraically closed.
Consider $L := \sfq(U)$, the two-dimensional linear subspace of
$S^2E^\vee$.
First, we claim that $x$ is a singular point of $X$ if and only if
$x$ is a singular point of $Q=V(q)$ for some $q \in L$.
The tangent space $T_xX$ is the intersection of the tangent spaces
$T_xQ\cap T_xQ'$, for any two distinct quadrics $Q,Q'$ from $L$.
If $x$ is a singular point of $Q$, then $T_xQ = |E|$.
Thus $\textrm{codim}~T_xX \le 1$ and $x$ is a singular point of $X$.
If $x$ is a singular point of $X$, then either $T_xQ=|E|$ for some
quadric $Q$, or $T_xQ = T_xQ'$ for any two quadrics $Q,Q'$ from $L$.
In the first case, $x$ is a singular point of $Q$.
In the latter case, if $b,b'$ are the corresponding bilinear forms,
then the tangent spaces are defined by linear forms
$b(v,-)$ and $b'(v,-)$ where $v \in E$ represents $x$.
Since the linear forms are proportional by assumption,
some non-trivial linear combination $(\lambda b + \mu b')(v,-)$ is trivial.
The quadric corresponding to that linear combination
is therefore singular at $x$.
This establishes the claim. 

As an immediate consequence, if a quadric $Q$ in $L$ has corank $> 1$,
then $X$ is not smooth; indeed, in this case, the singular locus of $Q$
has non-trivial intersection with some other quadric in the pencil
since it is positive dimensional.

Assume $\sfq$ is regular; in other words, that $X$ is smooth.
Suppose $q_0$ is a degenerate quadratic form in $L$ with corresponding
quadric $Q_0$.
Since $X$ is smooth, $Q_0$ has corank $\le 1$, so $q_0$ must be in $\sfD_0$.
Let $v_0 \in E$ represent the unique singular point $|v_0|$ of $Q_0$.
Since $|v_0| \notin X$, the space 
\[ |\{ q \in L\ :\ q(v_0)=0\}| \]
is just the point $|q_0|$.
By Lemma~\ref{lem:reg}, we conclude that the tangent space of $|D \cap L|$
at the point $|q_0|$ is $0$-dimensional.
Thus $|L|$ and $|D|$ meet transversally at each intersection point and
$|L \cap D|$ consists of exactly $n$ distinct points as desired.

Now, suppose $|L \cap \sfD|$ consists of exactly $n$ distinct points.
This is equivalent to $|L|$ intersecting transversally the
hypersurface $|\sfD|$.
From the proof of Lemma~\ref{lem:reg}, the differential of
$\tilde{\sfD} \to \sfD$ at a point $(q,v)$
has kernel equal to the radical of the associated bilinear form $b$.
The dimension of this radical is minimal precisely for $\sfD_0$;
so $\sfD_0$ is the smooth locus of $\sfD$.
Since $|L|$ intersects $|\sfD|$ transversally at every intersection
point, in fact $|L \cap \sfD| = |L \cap \sfD_0|$.
If $q_0 \in L \cap \sfD$, then, by Lemma~\ref{lem:reg},
$|\{ q \in L\ :\ q(v_0)=0\}|$ is a single point.
Thus, the unique singular point of the quadric $Q_0$ corresponding to
$q_0$ does not lie in $X$.
Since none of the singular points of the quadrics lie in $X$,
we conclude from above that $X$ is smooth.
\end{proof}

\section{Normal forms}
\label{sec:normalForms}

\subsection{Kronecker basis}
In this section, we prove Theorem~\ref{thm:NormalForm} and discuss some
of its consequences.
At least in retrospect, the theorem follows quite easily using known
techniques regarding pairs of bilinear forms going back
to L. Kronecker.
Throughout $E$ is a vector space of dimension $n=2m+1 \ge 3$ over a
field of characteristic $2$.

Given a pair of alternating bilinear forms $(b_0,b_1)$ on $E$,
we say that a direct sum $E = E_1 \oplus E_2$ is \emph{orthogonal} if it is
orthogonal with respect to both $b_0$ and $b_1$.
A pair of alternating bilinear forms is \emph{nonsingular}
if the corresponding determinant polynomial is non-zero.
For a positive integer $k$,
a \emph{basic singular pair} is a pair of alternating bilinear forms
$(b_0,b_1)$
such that there exists a basis
$\calB = (w_0,\ldots,w_k,v_0,\ldots,v_{k-1})$
of the underlying vector space
such that
\begin{equation} \label{eq:basicEquations}
\begin{array}{rcl c rcl c rcl}
b_0(w_i,w_j)&=&0, & &
b_0(v_i,v_j)&=&0, & &
b_0(w_i,v_j)&=&\delta_{i(j+1)},\\
b_1(w_i,w_j)&=&0, & &
b_1(v_i,v_j)&=&0, & &
b_1(w_i,v_j)&=&\delta_{ij},
\end{array}
\end{equation}
for all valid $i,j$.
We call a basis $\calB$ as above a \emph{Kronecker basis}.

\begin{thm} \label{thm:Kronecker}
Let $(b_0,b_1)$ be a pair of alternating bilinear forms on a vector
space $E$.  Then $E$ can be written as an orthogonal direct sum of a
nonsingular pair and a set of basic singular pairs.
\end{thm}

Theorem~\ref{thm:Kronecker} is a special case of
Theorem~3.3(1)~of~\cite{LS99}.
The statement and its proof are essentially due to Kronecker,
although his version
applies to symmetric bilinear forms in characteristic $\ne 2$;
see Theorem~3.1~of~\cite{Wat76}.
The proof in \cite{Wat76} carries over to the case of alternating bilinear
forms in characteristic $2$ as observed before the statement of
Theorem~9~of~\cite{Wat77}.

\begin{corollary}
Let $(q_0,q_1)$ be a regular pair of quadratic forms and let $(b_0,b_1)$
be the associated pair of alternating bilinear forms.
Then $(b_0,b_1)$ is a basic singular pair and the vector 
\begin{equation} \label{eq:OmegaDef}
\Omega = \sum_{i=0}^m \lambda^{m-i} \mu^i w_i
\end{equation}
spans $\rad(\lambda b_0 + \mu b_1)$ where
$w_i$ are basis elements in a Kronecker basis
\eqref{eq:basicEquations}.
\end{corollary}

\begin{proof}
By a straightforward calculation (see Lemma~3.2~of~\cite{LS99}),
for a basic singular pair $(b'_0,b'_1)$
on a vector space of dimension $2k+1$, the radical of
$\lambda b'_0 + \mu b'_1$
is spanned by the vector
\[
\Omega' = \sum_{i=0}^k \lambda^{k-i} \mu^i w_i \ .
\]

By Corollary~\ref{cor:corank}, any non-zero bilinear form from the family
$\lambda b_0 + \mu b_1$ has corank $1$.
Thus there is exactly one basic singular pair in the decomposition from
Theorem~\ref{thm:Kronecker} and its radical is given by
$\Omega=\Omega'$.
It remains only to establish that $k=m$, and thus there is no
nonsingular summand in the decomposition.

The vector $\Omega$ can be viewed as a homogeneous polynomial function
$U \to E$ of degree $m$.
Thus, if $k < m$ then in any choice of basis
the entries of $\Omega$ must all be divisible by a homogeneous polynomial
$g(\lambda,\mu)$ of positive degree.
But $\Delta(u)=\sfq_u(\Omega(u))$ for all $u \in U$, so the half-discriminant
would then be divisible by $g(\lambda,\mu)^2$.
This contradicts that the roots of
$\Delta$ are distinct and so $k=m$ as desired.
\end{proof}

It follows that the map $\Omega : U \to E$ is injective and its image
spans the same subspace $W$ as that spanned by $w_0, \ldots, w_m$.
Since $\Omega$ does not depend on any choice of coordinates, we see that
$W$ is canonical.
Any $(m+1)$-dimensional totally isotropic subspace of an alternating
bilinear form of rank $2m$ on a $(2m+1)$-dimensional space must contain
the radical.
Since $W$ is the minimal space containing the radical of the associated
bilinear form to every quadratic form in the pencil, we obtain the
following\footnote{The existence of such a subspace for any pencil of 
alternating forms is a known fact,
see the Lemma in the Appendix to the paper \cite{Beauville}}:

\begin{corollary}
The space $W$ spanned by $w_0, \ldots, w_m$ is canonical.
It is the unique common $(m+1)$-dimensional totally singular subspace of all
quadratic forms in the pencil.
\end{corollary}

Note that choice of basis $w_0, \ldots, w_m$ is determined by the choice
of basis $u_0, u_1$ in $U$ (or, equivalently, the choice of pair
$(q_0,q_1)$ corresponding to the pencil $|\sfq|$).
We also see that a choice of a basis $u_0,u_1$ in $U$ and the canonical
isomorphism $W$ with $(S^mU^\vee)^\vee$ defines a basis in $W$ equal to
the image of the dual basis to the monomial basis
$(t_0^m,t_0^{m-1}t_1,\ldots,t_1^m)$ in $S^mU^\vee$. 

Another consequence of the Kronecker theorem is that the map
$\Omega:U\to E$ is equal to the composition of the map $U\to
(S^mU^\vee)^\vee \to E$, where the second map is injective.
This implies that the image of $U$ is equal to the affine cone over a
\emph{Veronese curve} $R_m$ of degree $m$ in
$|W| \cong |(S^mU^\vee)^\vee| \cong \bbP^m$.
Recall that a Veronese curve in $\bbP^m$ is a smooth rational curve of
degree $m$ equal to the image of $\bbP^1$ under a map given by linearly
independent homogeneous polynomials of degree $m$.
By choosing a monomial basis $t_0^{m-i}t_1^i$ of such polynomials, it is
projectively equivalent to a curve $R_m$ in $\bbP^m$ given by equations
expressing the condition that 
\begin{equation} \label{eq:veronese}
\rank\begin{pmatrix}x_0&x_1&\ldots&x_{m-1}\\
x_1&x_2&\ldots&x_m\end{pmatrix} = 1.
\end{equation}
For example, when $m = 2$, this is a smooth conic in $\bbP^2$.
In a coordinate-free way, the Veronese curve is the image of $\bbP^1 =
|U|$ in $|(S^mU^\vee)^\vee|$ under the complete linear system
$|\calO_{|U|}(m)| = |S^mU^\vee|$.
The image of the scheme $V(\Delta)$ of roots of the half-discriminant
$\Delta$ of the pencil under the map $\Omega:|U|\to |W|$ is a
$0$-dimensional closed subscheme $Z$ of the Veronese curve $R_m$ of
length $n$. Over an algebraically closed extension of $\bbk$ it becomes
a union of $n$ distinct points.

We now prove the main theorem:
 
\begin{proof}[Proof of Theorem \ref{thm:NormalForm}]
Choose a Kronecker basis $w_0,\ldots,w_m,v_0,\ldots,v_{m-1}$ and the
corresponding coordinates $x_0,\ldots,x_m,y_0,\ldots,y_{m-1}$ in $E$.
Let $W$ be the span of $w_0, \ldots, w_m$ and let $L$ be the span
of $v_0, \ldots, v_{m-1}$.
Thus $E = W \oplus L$ is a direct sum of totally singular subspaces
for all quadrics in the pencil $| \sfq |$
(since they are totally isotropic for the associated alternating
bilinear forms).
This implies that the restrictions of $q_0,q_1$ to $W$ (resp. $L$) is
equal to a linear combination of squares of $x_i$'s (resp. $y_i$'s).
So, we have reduced $q_0,q_1$ to the expressions from the Theorem.
It remains to prove the assertion about the coefficients $\{ a_i \}$.
We use the expression of $\Omega$ computed above to
find the half-discriminant.
Let $u_0,u_1$ be the basis of $U$ corresponding to $q_0,q_1$ and
$t_0,t_1$ be the dual basis.
Then 
\begin{align*}
\Delta(t_0,t_1)
=&\ (t_0q_0+t_1q_1)(t_0^m,t_0^{m-1}t_1,\ldots,t_1^m,0, \ldots, 0)\\
=&\ t_0 \left(\sum_{i=0}^ma_{2i}t_0^{2(m-i)}t_1^{2i}\right)
+t_1\left( \sum_{i=0}^ma_{2i+1}t_0^{2(m-i)}t_1^{2i} \right) \\
=&\ \sum_{i=0}^{2m+1}a_it_0^{2m+1-i}t_1^i
\end{align*}
as desired.
\end{proof}

We say that a regular pair of quadratic forms is in \emph{normal form}
if it is written in coordinates corresponding to a Kronecker basis in $E$. 

Note that, given a specific choice of pair $(q_0,q_1)$, the
basis $w_0, \ldots, w_m$ in a Kronecker basis is canonical.
However, given only a pencil $| \sfq |$, only 
the vector space $W$ spanned by $w_0,\ldots,w_m$ is canonical.
This has some strong consequences which we discuss now.
In the following $X$ denotes the base scheme $\Bs(\sfq)$ of a regular
pencil of quadrics $|\sfq|$ in $|E|$.

\begin{thm}\label{thm:canPlane}
Assume $m \ge 2$.
If $\bbk$ is a perfect field, then the variety $X$ contains
a canonical $(m-2)$-plane $\Pi$ defined over $\bbk$.
In particular, a del Pezzo surface of degree $4$ over a perfect field of
characteristic $2$ has a canonical point.
\end{thm}

\begin{proof}
Note that the ideal of the subscheme $X \cap |W|$ is generated by
\[
q_0|_W = \sum_{i=0}^m a_{2i}x_i^2, \ 
q_1|_W = \sum_{i=0}^m a_{2i+1}x_i^2
\]
in projective coordinates for $|W|$.
Over a perfect field, the radical ideal is
$\langle l_0, l_1 \rangle$ where
$l_0^2=q_0|_W$ and $l_1^2=q_1|_W$.
If $l_0 = c l_1$ for some $c \in \bbk$,
then $a_{2i}=c^2a_{2i+1}$ for all $i$ in $0, \ldots, m$.
But this implies that
\[ \Delta(T,1)=(T+c^2)g(T^2) \]
for some polynomial $g \in \bbk[T]$,
contradicting the separability of $\Delta(T,1)$.
Thus the subspace $\{ l_0=l_1=0 \}$ in $X$ has codimension $2$ in
$|W|$ as desired.
\end{proof}

The vector space $L$ spanned by $v_0,\ldots,v_{m-1}$ is not canonical;
in fact, we have the following.
 
\begin{lem}\label{lem:LgivesNF}
Let $(q_0,q_1)$ be a regular pair of quadratic forms and let $(b_0,b_1)$
be the associated bilinear forms.
Let $W$ be the canonical subspace spanned by $(w_0,\ldots,w_m)$ in a
Kronecker basis.
Suppose $L$ is a totally isotropic subspace for both $b_0$ and $b_1$
and that $E = W\oplus L$.
Then there exists a basis $(v_0,\ldots,v_{m-1})$ in $L$ that, together with
the canonical basis $(w_0,\ldots,w_m)$ in $W$, forms a Kronecker
basis in $E$.
\end{lem}

\begin{proof}
Since $b_1$ has a $1$-dimensional radical spanned by $w_m$,
the alternating bilinear form induced on the quotient space
$\bar{E} = E/\langle w_m \rangle$ is a direct sum of $m$ hyperbolic planes.
Since $w_m \in W$, the
images of $W$ and $L$ in $\bar{E}$ are complementary totally isotropic
subspaces of $\bar{E}$; and thus provide a pair of
complementary maximal totally isotropic subspaces.
Thus, we may find a basis $v_0, \ldots, v_{m-1}$ for $L$ such that it
satisfies the desired equations for $b_1$ from
\eqref{eq:basicEquations}.

We have established all desired equations from \eqref{eq:basicEquations}
except for $b_0(w_i,v_j)=\delta_{i(j+1)}$.
We know that some Kronecker basis $(w_0,\ldots,w_m,v'_0,\ldots,v'_{m-1})$
exists with corresponding splitting $E = W \oplus L'$.
For any $v\in L$ we can write $v = w+v'$ for $w\in W, v'\in L'$.
For all $i = 0, \ldots, m$, we have
$b_0(w_i,v)=b_0(w_i,v')$ since $b_0(w_i,w)=0$; and similarly for $b_1$.
From \eqref{eq:basicEquations}, we have that
$b_1(w_i,v'_j) = \delta_{ij} = b_0(w_{i+1},v'_j)$ for all valid $i,j$.
From this we conclude that $b_1(w_i,v) = b_0(w_{i+1},v)$ for any
$v \in L$.
We conclude that $b_0(w_i,v_j) = b_1(w_{i-1},v_j)=\delta_{i(j+1)}$
for all $i,j$ except when $i=0$.
This last case follows from the fact that $w_0$ is in the radical of
$b_0$.
\end{proof}

We have identified the radical subspace $W$ with $(S^mU^\vee)^\vee$.
The next proposition tells that we can identify any complement $L$
as in Lemma~\ref{lem:LgivesNF} with the space $S^{m-1}U^\vee$.

\begin{prop} \label{prop:coordFree}
Let $(q_0,q_1)$ be a regular pair of quadratic forms on $E$.
Up to a choice of volume form on $E$,
there is a canonical isomorphism $\iota : U \to U^\vee$ and
a canonical isomorphism $L \cong S^{m-1}(U^\vee)$, such that
\begin{equation} \label{eq:form1}
\sfb_u(v,w) = \psi_1(\iota(u)f_2) + \psi_2(\iota(u)f_1)
\end{equation}
where $v = (\psi_1,f_1)$ and $w = (\psi_2,f_2)$ are elements of
$W\oplus L = (S^mU^\vee)^\vee\oplus S^{m-1}U^\vee$ and
we note that $\iota(u) f_i$ can be viewed as an element of $S^m(U^\vee)$
via the multiplication map $U^\vee \otimes S^{m-1}U^\vee \to S^mU^\vee$.
\end{prop}

\begin{proof}
Let $u_0,u_1$ be the basis for $U$ such that
$\sfq_{u_0}=q_0$ and $\sfq_{u_1}=q_1$,
and let $t_0,t_1$ be the dual basis.
The map $\iota$ is defined to satisfy $t_0=\iota(u_1)$ and $t_1=\iota(u_0)$.
We use the canonical isomorphism between $W$ and $(S^mU^\vee)^\vee$
which we deduced from formula \eqref{eq:OmegaDef};
more concretely: $w_i \mapsto (t_0^{m-i} t_1^i)^\vee$
(this is where we use the volume form on $W$).
Consider the isomorphism $L \cong S^{m-1}(U^\vee)$ given
by $v_i \mapsto t_0^{(m-1)-i}t_1^{i}$ for $i$ in $0,\ldots,m-1$.
A direct comparison of \eqref{eq:form1} and \eqref{eq:basicEquations}
in coordinates for $U$, $W$, and $L$ establishes the result.
\end{proof}

Proposition~\ref{prop:coordFree} allows us to compare normal forms for
different choices of pairs in the same pencil.

\begin{lem} \label{lem:changePair}
Suppose $u, u' \in U$ are linearly independent and consider
a Kronecker basis $(w_0,\ldots,w_m,v_0,\ldots, v_{m-1})$
for the pair $(q_u,q_{u'})$.
For any $g \in \GL(U)$, there exists an $h \in \GL(W) \times \GL(L)$
such that $(h(w_0), \ldots, h(v_{m-1}))$ is a Kronecker basis for
$(\sfq_{g(u)},\sfq_{g(u')}))$.
\end{lem}

\begin{proof}
There exists some $c \in \bbk$ such that
$\iota(gu)=c(g^\vee)^{-1}\iota(u)$ for all $u \in U$.
Using the isomorphisms from Proposition~\ref{prop:coordFree},
consider
\[
h = \left((S^mg^\vee)^\vee, c^{-1} S^{m-1}(g^\vee)^{-1}\right)
\] in $\GL(W) \times \GL(L)$.
Using the shorthand notation $g_*=(S^kg^\vee)^\vee$ and
$g^*=S^{k-1}(g^\vee)$ for all positive integers $k$,
we find
\begin{align*}
& (h\psi)(\iota(gu)(hf)) = 
(g_*\psi)((\iota(gu))(c^{-1} (g^*)^{-1}f))\\
=&\ \psi( g^* ((g^*)^{-1}\iota(u))((g^*)^{-1}f)) 
= \psi( g^* (g^*)^{-1}(\iota(u)f))
= \psi(\iota(u)f)
\end{align*}
for all $\psi \in W \cong (S^mU^\vee)^\vee$,
$f \in L \cong S^{m-1}(U^\vee)$, and $u \in U$.
From this we conclude that
\[ \sfb_{gu}(hv,hw)=\sfb_u(v,w) \]
for all $v,w \in E$ and $u \in U$.
Thus, $(\sfb_{gu},\sfb_{gu'})$ is in Kronecker normal form
relative to the basis $h(w_0), \ldots, h(v_{m-1})$.
\end{proof}

\begin{remark} \label{rem:technicalNuisance}
It is frequently convenient to assume that $a_n \ne 0$ in the
half-discriminant polynomial or, equivalently, that
the quadratic form $q_1$ is nondegenerate.
Lemma~\ref{lem:changePair} shows that this is often a harmless
assumption.
\end{remark}

\subsection{The quasi-split case}
\label{sec:quasiSplit}

Following the classical terminology in geometry of ruled varieties, a
\emph{generator} of the variety $X$ is a linear subspace of $X$ of maximal
possible dimension (over the algebraic closure).
In the spirit of \cite{Skor}, we say that $X$ is \emph{quasi-split} if
$X = \Bs(\sfq)$ has a generator defined over $\bbk$.
We say $X$ is \emph{split} if all generators of $X$ are defined over
$\bbk$.
We shall see in Corollary~\ref{cor:generators} below that,
when $\bbk$ is algebraically closed,
every $X$ has exactly $2^{2m}$ generators.
(For two general quadrics, this also follows by
Example~14.7.15~of~\cite{Fulton}.)

\begin{thm}\label{thm:genExists}
A generator of $X$ has dimension $m-1$.
If $\bbk$ is algebraically closed, then $X$ contains a generator.
\end{thm}

\begin{proof}
Assume $\bbk$ is algebraically closed.
Let $F_r(X)$ be the subscheme of the Grassmannian $G_r(\bbP^{n-1})$
of $r$-dimensional subspaces in $\bbP^{n-1}$ which are contained in $X$.
From Theorem~2.1~of~\cite{Debarre}, we know $F_r(X)$ is non-empty
if and only if $r \le m-1$.  Thus, $X$ has a generator
(which is of dimension $m-1$).
\end{proof}

\begin{remark}
We shall see below that all $X$ are quasi-split precisely when
$\bbk$ is closed under separable quadratic extensions.
\end{remark}

This brings us to the main theorem of this section.
Since $X$ is always quasi-split over an algebraically closed field,
Theorem~\ref{thm:geometricRis0} follows immediately from the following.

\begin{thm}\label{thm:qsGivesr0}
$X$ is quasi-split if and only if there exists a normal form for the pencil 
$\sfq$ with $r_1= \ldots = r_{2m} = 0$.
\end{thm}

\begin{proof}  
Note that a generator $\Lambda$ is equal to $|L|$, where $L$ is totally
isotropic for all quadrics in the pencil.
When $r_i =0$ in the normal form defined by a Kronecker basis
$(w_0,\ldots,w_{m},v_0,\ldots,v_{m-1})$,
the span of $v_0,\ldots,v_{m-1}$ provides such a subspace.
Conversely, given such a subspace $L$,
Lemma~\ref{lem:LgivesNF} provides the desired normal form
under the condition that $L$ is complementary to $W$.

Thus, the difficulty is to prove that $|L| \cap |W| = \emptyset$
(in other words, that $L \cap W = 0$).
We may assume without loss of generality, $\bbk$ is algebraically
closed.

Let $\Pi = |W| \cap X$ be the canonical subspace  from
Theorem~\ref{thm:canPlane}.
Let $S =  |L|\cap |W| = |L| \cap \Pi$.
Suppose $\dim~S = k\ge 0$.
Let $P$ be the span of $|W|$ and $|L|$, its dimension is equal to
$2m-k-1$.
Let $Q = V(q)$ be a quadric from the pencil.
The restriction of $Q$ to $|W|$
is a linear subspace $Y$ of dimension $m-1$ (taken with multiplicity 2).
The quadric $Q$ also contains the linear subspace $|L|$ of
the same dimension.
They intersect along the $k$-dimensional subspace $S$.
We claim that $Q$ must be singular along $S$.

To see this, choose coordinates $x_0,\ldots,x_{2m-k-1}$ for $P$ in
such a way that $|L|$ is given by $x_{m}=\ldots = x_{2m-k-1} = 0$ and
$|W|$ is given by $x_{0}=\ldots = x_{m-k-2} = 0$, so that $S$ is given
by $x_{i}= 0, i\ne m-k-1,\ldots,m-1$.
The restriction $Q|_P$ of $Q$ to $P$ must have
an equation of the form
\[
L(x_{m},\ldots,x_{2m-k-1})^2+\sum_{i=0}^{m-k-2}x_iM_i(x_m,\ldots,x_{2m-k-1})
= 0
\]
where $L$ and $M_1, \ldots, M_{m-k-2}$ are linear homogeneous
polynomials.
Since the polynomial defining this equation does not contain the
variables $x_i, i\ne m-k-1,\ldots,m-1$, the quadric is singular along
the subspace $S$. In other words, the subspace $P$ is tangent to the
quadric $Q$ along $S$. Thus any hyperplane containing $P$ is tangent to
$Q$ along $S$. Since $S$ and $P$ do not depend on $Q$, we can find a
hyperplane tangent to any $Q$ along $S$. Since the tangent space of $X$
at any point is of codimension 2, and it is equal to the intersection of
tangent hyperplanes of quadrics in the pencil, we find a contradiction
with the assumption that $X$ is
smooth.
\end{proof}

The proof of the preceeding theorem has the following useful
consequence:

\begin{corollary} \label{cor:generatorIsL}
If $\Lambda$ is a generator of $X$, then there is a Kronecker normal
form with decomposition $E = W \oplus L$ such that $|L|=\Lambda$.
\end{corollary}

\subsection{Applications to rationality}
\label{sec:applications}

This canonical subspace $\Pi$ from Theorem~\ref{thm:canPlane}
allows us to prove Corollary~\ref{cor:Xrational}
from the introduction:

\begin{proof}[Proof of Corollary~\ref{cor:Xrational}]
 First we consider the case $m=2$, where we are dealing with quartic
del Pezzo surfaces in $|E|\cong \bbP^4$.
The canonical $(m-2)$-plane $\Pi$ is thus a canonical point.
It is known that any del Pezzo surface of degree $\ge 3$ with a rational
point is unirational over the base field (see
Theorem~3.5.1~of~\cite{MT}).

Now we consider the case where $m > 2$. Consider the projection
$p:X\setminus \Pi\to\bbP^{m+1}$ with center at $\Pi$. Since $\Pi$ is
defined over $\Bbbk$, the projection is defined over $\Bbbk$.
The morphism $p$ extends to a morphism $p':X_\Pi\to \bbP^{m+1}$
where $X_\Pi$ is the blow-up of $X$ at center $\Pi$.

Let $x$ be a point in $X \setminus \Pi$.
Let $F_x$ be the closure of $p^{-1}(p(x))$ in $\bbP^{n-1}$.
We claim that $F_x$ is the span $\la x, \Pi \cap T_x X \ra$
where $T_x X$ is the embedded tangent space to $X$ at $x$.
Let $v \in E$ be a vector such that $x = [v]$ and let
$e_0,\ldots,e_{m-2}$ be a basis of $\Pi$.
The subvariety $F_x$ is equal to the closure of the
intersection of $X \setminus \Pi$
and the span $\la \Pi,x\ra$ of $\Pi$ and $x$.
The span $\la \Pi,x\ra$ is equal to $[tv+\sum_{i=0}^{m-2} s_ie_i]$,
where $[s_0,\ldots,s_{m-2},t]$ are projective coordinates
in $\la \Pi,x\ra$.
Plugging in the equations $q_0 = q_1 = 0$ of $X$ we get,
for $k = 0,1$, that
\begin{align*}
&q_k\left(tv+\sum_{i=0}^{m-2} s_ie_i\right)\\
=&
t^2q_k(v)+
q_k\left(\sum_{i=0}^{m-2} s_ie_i\right)
+t\sum_{i=0}^{m-2}s_ib_k(e_i,v) \\
=& t\sum_{i=0}^{m-2}s_ib_k(e_i,v).
\end{align*}
This shows that $F_x$ is the span
$\la x, \Pi\cap T_x(X)\ra$ as claimed.

First, assume that $\bbk$ is algebraically closed.
We know that if $\Pi\cap T_x(X)$ is non-empty then $\dim \Pi\cap T_x(X)$ is
$m-2-r(x)$, where $r(x)$ is the rank of the matrix
\[
\begin{pmatrix}b_0(e_1,v)&\ldots& b_0(e_{m-2},v)\\
b_1(e_1,v)&\ldots &b_1(e_{m-2},v)\end{pmatrix}.
\]
Let $C$ be the set of points $x \in X$ where $r(x) < 2$:
these are the points such that, for some quadric $Q$
given by $\lambda q_0+\mu q_1=0$,
the associated bilinear form $b$ satisfies
$b(w,v)=0$ for all $w \in E$ representing $[w] \in \Pi$.
Suppose $C$ is equal to all of $X$.
By Theorem \ref{thm:genExists}, the space $X$ contains a
$(m-1)$-dimensional projective subspace $\Lambda$.
Choose a point $x$ from $\Lambda$.
Note that the tangent space $T_x(Q)$ is defined by the fact that
$b(w,v)=0$ for all $w \in E$.
Thus $T_x(Q)$ contains the $(m-1)$-dimensional projective
subspace $\Lambda$ and the $(m-2)$-dimensional projective subspace
$\Pi$.
Since $\Lambda\cap \Pi = \emptyset$, we obtain that $T_x(Q) \cap Q$
contains a projective subspace of dimension $2m-2$.
However, by Corollary~\ref{cor:corank}, the corank
of any quadric in the pencil is equal to $1$.
Hence $Q$ cannot contain
projective subspaces of dimension greater than $m-1$.
This contradicts $C$ being the whole space.
Thus, for a general point $x\in X$, we have $r(x) = 2$.

Consequently, $\Pi\cap T_x(X)$ is empty if $m\le 3$, or
of dimension $\ge m-4$ if $m\ge 4$.
This implies that $F_x=\la x, \Pi\cap T_x(X)\ra$
consists of $x$ if $m\le 3$ and
$\dim(F_x) = m-3$ if $m\ge 4$.
If $m = 2$, $X$ is a del Pezzo surface of degree
$4$ and the projection $p$ is a birational map onto a cubic surface in
$\bbP^3$.
If $m = 3$, we obtain that the projection $p$ is a birational map
onto $\bbP^4$.
If $m > 3$, we obtain that the general fiber $(p')^{-1}(p(x))$ of the
projection map $p' : X_\Pi\to \bbP^{m+1}$ is isomorphic to the blow-up of
the $(m-3)$-dimensional subspace $\la x, \Pi\cap T_x(X)\ra$
along the $(m-4)$-dimensional subspace $\Pi\cap T_x(X)$.
Thus the general fiber is a projective space of dimension $m-3$.
Thus, $X$ is birationally isomorphic to a projective $(m-3)$-bundle over
$\bbP^{m+1}$.

Now, assume $\Bbbk$ is any perfect field. Since the projection is a
$\bbk$-rational map, we immediately get that $X$ is rational if $m =
3$.
It remains to consider $m>3$.
By construction, $X_\Pi$ is a subvariety of a projective bundle $\bbP(E)$
over $\bbP^{m+1}$.
Moreover, the general fibers of $X_\Pi \to \bbP^{m+1}$ are linear
subspaces of the fibers of $\bbP(E) \to \bbP^{m+1}$.
Thus, over some open subscheme $S$ of $\bbP^{m+1}$, the subvariety $X_\Pi$
restricts to a projective bundle over $S$
(not just a Severi-Brauer scheme over $S$).
Thus $X_\Pi$ is rational.
\end{proof}

Yu. Prokhorov asked us for a geometric interpretation of the canonical
point $\Pi$ in the plane model of a del Pezzo surface of degree 4.
We consider this question in Remark~\ref{rem:ptInPlaneModel}.

\section{Isomorphisms of normal forms}
\label{sec:NormalFormIsos}

The main goal of this section is to prove
Theorem~\ref{thm:OtherNormalForms}.
In order to do this, we need to study isomorphisms between pairs of
quadratic forms in Kronecker normal form.
The associated bilinear form of any such pair is always the same, so
any isomorphism of pairs of quadratic forms is an automorphism
of a basic singular pair of alternating bilinear forms.

Thus, our first task is to compute the automorphisms of a basic singular
pair of alternating bilinear forms.
Next, we collect some results about the finite dimensional algebra
$A$ from the introduction.
We then use this algebra to determine which automorphisms of basic
singular pairs correspond to which isomorphisms of pairs of
quadratic forms.

\subsection{Automorphisms of basic singular pairs}

Let $(b_0,b_1)$ be a basic singular pair of alternating bilinear forms
(for example, obtained from a regular pair of quadratic forms),
and let $w_0, \ldots, w_m, v_0, \ldots, v_{m-1}$ be a Kronecker basis
for the underyling vector space $E$ of dimension $n=2m+1$.

An \emph{automorphism} of the pair $(b_0, b_1)$
is an element $g \in \GL(E)$ such that
\[ b_0(gv,gv)=b_0(v,v)
 \textrm{ and }
b_1(gv,gv)=b_1(v,v) \]
for any $v \in V$.

\begin{lem} \label{lem:autoAlt}
The group of automorphisms $\Aut(b_0,b_1)$ of the pair $(b_0,b_1)$
consists of elements $g\in \GL(E)$ of the form 
\begin{equation} \label{eq:gs}
g(w_i) = w_i,\ g(v_i) = v_i+\sum_{k=0}^{m}s_{i+k}w_{k}
\end{equation}
where $s_0, \ldots, s_{n-2}$ are any elements in $\bbk$.
\end{lem}

Explicitly, the automorphisms have the form
\[
g = \begin{pmatrix}I_{m+1}&S\\
0_{m,m+1}&I_{m}\end{pmatrix},
\]
where
\[
S = \begin{pmatrix}
s_0&s_1&\ldots&s_{m-1}\\
s_1&s_2&\ldots&s_m\\
\vdots&\vdots&\vdots&\vdots\\
s_m&s_{m+1}&\ldots&s_{2m-1}
\end{pmatrix}
\]
is a \emph{catalecticant matrix}
$\textrm{Cat}_{m-1}(s_0,\ldots,s_{2m-1})$ (see \cite{CAG}, 1.4.1).

\begin{proof}
One checks directly that the given elements
are automorphisms of $(b_0,b_1)$ via \eqref{eq:basicEquations}.
It remains to check that these are the only automorphisms.

Let $g$ be an automorphism of $(b_0,b_1)$.
Note that $g$ must fix each $w_i$ since they are canonical.
It remains to consider $v_0, \ldots, v_{m-1}$.
Thus
\[ g(v_i)=h(v_i)+\sum_{k=0}^m l_{ik}w_k  \]
for some $h \in \GL(L)$ and elements $l_{ik}$ in $\bbk$.
Since
\[
\delta_{ij}=b_1(gv_i,gw_j)=b_1(hv_i,w_j)
\]
we conclude that $h$ is the identity.
Since
\[
0=b_1(gv_i,gv_j)=
\sum_{k=0}^m l_{jk} b_1(v_i,w_k)
+\sum_{k=0}^m l_{ik} b_1(w_k,v_j)
= l_{ij} + l_{ji}
\]
we conclude that $l_{ij}=l_{ji}$.
Similarly, using $b_0$ we conclude that $l_{j(i+1)}=l_{i(j+1)}$.
We obtain that $l_{i(j+1)}=l_{(i+1)j}$ and thus the values of
$l_{ij}$ depend only on the sum of their indices: $l_{ik}=s_{i+k}$.
\end{proof}

\subsection{Results on \texorpdfstring{$\bbk$}{k}-algebras}

Here we collect some facts about finite $\bbk$-algebras which we will
need later.
All are standard results except for Lemma~\ref{lem:BbasisSquare}. 

Consider a separable polynomial
\[ f(T) = a_n T^n + \cdots +a_1 T + a_0 \]
in $\bbk[T]$ of degree $n$.
Consider the $\bbk$-algebra $A = \bbk[T]/(f(T))$ and let $t$ be the image
of $T$ under the map $\bbk[T] \to A$.

Note that since $f$ is a separable polynomial, we may also write
$A$ as a direct sum $A = A_1 \oplus \cdots \oplus A_l$ where each
$A_i$ is a separable field extension of $\bbk$.
Indeed, we have isomorphisms $A_i \cong \bbk[T]/(f_i(T)) \cong A/(f_i(t))$
where $f_i$ is an irreducible polynomial dividing $f$.

Define polynomials $g_1, \ldots, g_l$ in $\bbk[T]$ via
\[ f(T) = f_i(T)g_i(T) \ . \]
Since the polynomials $g_i$ are coprime, we may write
\begin{equation} \label{eq:epsilon}
1 = \epsilon_1 + \cdots + \epsilon_l
\end{equation}
where $\epsilon_i \in ( g_i(t) )$ and they satisfy relations
$\epsilon_i \epsilon_j = 0$ when $i\ne j$,
and $\epsilon_i^2 = \epsilon_i$.
In other words, $\epsilon_1, \ldots, \epsilon_l$
form an orthogonal set of idempotents in $A$.

Multiplication by an idempotent $\epsilon_i$ defines a homomorphism
from $A$ into the corresponding summand $A_i$.
Specifically, if $h$ is any polynomial in $\bbk[T]$ we have
\begin{equation} \label{eq:idemAction}
\epsilon_i h(t) \equiv h(t) \mod (f_i(t))
\end{equation}
in $A_i \cong A/(f_i(t))$.
If $f_i$ is of degree 1 with root $\alpha_i$,
then multiplication by $\epsilon_i$ corresponds to the $\bbk$-algebra
homomorphism $A \to A_i \simeq \bbk$ determined by $t \mapsto \alpha_i$.

For any $a\in A$, the linear map $x\mapsto ax$ is an endomorphism of the
vector space $A$ over $\bbk$, we denote by $\tr_{A/\bbk}(a)$ its trace.
The formula
\[ \la a,b\ra = \tr_{A/\bbk}(ab) \]
defines a symmetric bilinear form on $A$.
Its restriction to each summand $A_i$ is the usual trace for a separable
extension of fields.
Since $A$ is separable, the trace form is non-degenerate.
In particular, the natural homomorphism
$A\to A^\vee, a\mapsto (x\mapsto \tr_{A/\bbk}(ax))$
of vector spaces over $\bbk$ is a bijection.

If $f$ splits completely into linear factors
(for example, if $\bbk$ is separably closed),
then $A \cong \bbk^n$.
If $\alpha_1, \ldots, \alpha_n$ are the roots of $f$,
then a canonical isomorphism $A \cong \bbk^n$ is defined via
\[
h(t) \mapsto \left( h(\alpha_1), \ldots, h(\alpha_n) \right)
\]
for any polynomial $h$ in $k[T]$.
In this case, the trace can be computed as
\[
\tr_{A/\bbk}\left(h(t)\right) = \sum_{i=1}^n h(\alpha_i) \ .
\]
By passing to a separable closure, one can compute the trace even when
$f$ does not split into linear factors over the original field.

Consider a basis of $A$ given by the elements 
\begin{equation} \label{eq:def_d_i}
d_i = a_{i+1}+a_{i+2}t+\cdots+a_nt^{n-1-i}
\end{equation}
for $i = 0,\ldots,n-1$.
Using the recurrence relation $d_i=t d_{i+1}+a_{i+1}$,
we get the identity in $A[X]$
\begin{equation}\label{eq:GenFunc_def_u}
f(X) = (X-t)(d_0 + d_1 X + \cdots + d_{n-1}X^{n-1}) \ .
\end{equation}
Note that $d_{n-1}$ spans the ``constant'' subalgebra
$\bbk \subset A$.

Let $f'(T)$ denote the formal derivative of $f(T)$.
Since $f(T)$ is a separable polynomial, $f'(T)$ is coprime to $f(T)$
and thus $f'(t)$ is an invertible element in $A$.

The following proposition shows that the elements $\frac{d_i}{f'(t)}$
form the dual basis of the basis $1,t,\ldots,t^{n-1}$
with respect to the trace form. 

\begin{prop} \label{prop:dualBasis}
$ \tr_{A/\bbk}\left(d_i\frac{t^j}{f'(t)}\right) = \delta_{ij}$
for $i,j$ in $0,\ldots, n-1$.
\end{prop}

\begin{proof}
In the case when $A$ is a field, this can be found in, for example,
Proposition~III.1.2~of~\cite{Lang94}.
We give the proof here since it is short and uses formulas that we we
will need later.
We may assume without loss of generality that $\bbk$ is separably
closed.

Let $\alpha_1, \ldots, \alpha_n$ be the roots of $f$.
We use the following \emph{Euler's identity}:
\begin{equation}\label{eq:Euler}
X^k - \sum_{i=1}^n \frac{f(X)}{X-\alpha_i}
\frac{\alpha_i^k}{f'(\alpha_i)} = 0, \quad k = 0, \ldots, n-1
\end{equation}
in the ring $A\llbracket X \rrbracket$.
To see this, one uses that
$f'(\alpha_i) = \prod_{j\ne i}(\alpha_i-\alpha_j)$,
then checks that the expression on the left-hand
side is a polynomial of degree $< n$ and has $n$ roots
$\alpha_1,\ldots,\alpha_n$, hence it must be zero.

One can extend the trace $\tr_{A/\bbk}$ to a $\bbk$-linear function
$A\llbracket X\rrbracket \to \bbk \llbracket X\rrbracket$
by applying the trace function to each coefficient of $X^i$.
Thus, using \eqref{eq:GenFunc_def_u}, we obtain 
\[
\tr_{A/\bbk}\left( \frac{f(X)}{X-t} \frac{t^k}{f'(t)} \right) = 
\tr_{A/\bbk}\left( \left( \sum_{i=0}^{n-1}d_iX^i \right)
\frac{t^k}{f'(t)}\right) = X^k,\]
for every $k$ in $0, \ldots, n-1$.
Comparing the coefficients, we get the assertion.
\end{proof}

We need an explicit formula for squaring an element with respect to the
basis $\{ d_0, \ldots, d_{n-1} \}$.

\begin{lem} \label{lem:BbasisSquare}
Given the equation
\[
r_0d_0+\ldots+r_{n-1}d_{n-1}=(s_0d_0+\ldots+s_{n-1}d_{n-1})^2
\]
for some coefficients $s_0,\ldots,s_{n-1}$ in $\bbk$,
we have
\[
r_k = \sum_{j=0}^{n-1}s_j^2a_{2j+1-k}
\]
for all $k$ in $0,\ldots, n-1$.
(We assume by convention $a_i=0$ when $i < 0$ or $i > n$.)
\end{lem}

\begin{proof}
Using Proposition~\ref{prop:dualBasis}, it suffices to show that
\[
\tr_{A/\bbk} \left( d_j^2\frac{t^k}{f'(t)} \right) = a_{2j+1-k}
\]
for all $j,k$ in $0,\ldots, n-1$.

Rearranging Euler's equation \eqref{eq:Euler} we obtain:
\begin{align*}
\frac{X^k}{f(X)}
&= \sum_{i=1}^n \frac{1}{X-\alpha_i} \frac{\alpha_i^k}{f'(\alpha_i)} \ .
\\
\intertext{Now, differentiating with respect to $X$, we have}
\frac{kX^{k-1}f(X)-X^kf'(X)}{f(X)^2}
&= \sum_{i=1}^n \frac{-1}{(X-\alpha_i)^2} \frac{\alpha_i^k}{f'(\alpha_i)}
\\
\frac{\partial}{\partial X} \left(X^kf(X)\right)
&= \sum_{i=1}^n \left(\frac{f(X)}{X-\alpha_i}\right)^2
\frac{\alpha_i^k}{f'(\alpha_i)} \ .
\\
\intertext{Using \eqref{eq:GenFunc_def_u}, we conclude}
\frac{\partial}{\partial X} \left( \sum_{i=0}^n a_iX^{k+i} \right)
&=
\tr_{A/\bbk} \left( \left(\sum_{j=0}^{n-1} d_j^2X^{2j} \right)
\frac{t^k}{f'(t)} \right)
\\
\sum_j a_{2j+1-k}X^{2j}
&=
\sum_{j=0}^{n-1} \tr_{A/\bbk} \left( d_j^2\frac{t^k}{f'(t)} \right) X^{2j}
\end{align*}
as desired.
Note that the assumption that $\bbk$ has characteristic $2$ is essential!
\end{proof}

\subsection{Isomorphisms of Pairs of Quadratic Forms}

Let $(q_0,q_1)$ be a regular pair of quadratic forms
in normal form with respect to a Kronecker basis
$\calB = (w_0,\ldots,w_m,v_0,\ldots,v_{m-1})$ for $E$.
Note that since $w_0, \dots, w_m$ are canonical,
the polynomial $\Delta$ and its coefficients
$a_0, \ldots, a_n$ are the same regardless of the choice of basis
$\calB$.
However, the coefficients $r_0, \ldots, r_{n-2}$ from
Theorem~\ref{thm:NormalForm} are more subtle.

Assume that $a_n \ne 0$ in $\Delta$ or, equivalently,
that $q_1$ is nondegenerate
(see Remark~\ref{rem:technicalNuisance}).
Define
\[ f(T) := \Delta(1,T) = a_0 + \cdots + a_nT^n \ . \]
Setting $A = \bbk[T]/(f(T))$, we have an algebra as in the previous
section.

The \emph{$r$-invariant} of the normal form is the element $r \in A$
given by
\begin{equation} \label{eq:rInvDef}
r = r_0 d_0 + \cdots + r_{n-2} d_{n-2}
\end{equation}
where $d_0, \ldots, d_{n-1}$ is the basis for $A$
defined in \eqref{eq:def_d_i}.
Note that the basis element $d_{n-1}$ spanning $\bbk$ does not appear in
the expression \eqref{eq:rInvDef},
so the set of possible $r$-invariants are in bijection with
the quotient space $A/\bbk$.

For $g \in \GL(E)$, let $(q'_0,q'_1)$ be the pair given by
\[
q'_0(v) = q_0(gv) \textrm{ and } q'_1(v) = q_1(gv)
\]
for $v \in E$.
Note that the half-discriminant polynomial of $(q'_0,q'_1)$
remains the same.
Let $r'$ be the the $r$-invariant of the new normal form.

We assume, without loss of generality,
that $(q'_0,q'_1)$ is also in Kronecker normal form with respect to the
basis $\calB$.
Consequently, both pairs have the same pair $(b_0,b_1)$ of
associated bilinear forms and we may take $g \in \Aut(b_0,b_1)$.

Any $g \in \Aut(b_0,b_1)$ is determined by the elements
$s_0, \ldots, s_{n-2}$ from \eqref{eq:gs}.
We construct an element
\begin{equation} \label{eq:sInvDef}
s = s_0 d_0 + \cdots + s_{n-2} d_{n-2}
\end{equation}
in $A$ as with the $r$-invariant above.

Reversing the process,
one checks that this gives rise to a group homomorphism
\begin{equation} \label{eq:phiDef}
\phi : A \to \Aut(b_0,b_1) \subset \GL(E)
\end{equation}
where $A$ is viewed as an additive group.
Note that the kernel of $\phi$ is the subgroup $\bbk$ since $d_{n-1}$
spans $\bbk$.

Theorem~\ref{thm:OtherNormalForms} is a consequence of the following.

\begin{thm} \label{thm:ArtinSchreierFormula}
If $g = \phi(s)$ for some $s\in A$, then
\[
r' \equiv r + \wp(s) \mod \bbk
\]
where $\wp : A \to A$ is the Artin-Schreier map $s \mapsto s^2 + s$.
\end{thm}

\begin{proof}
Recall that $g(w_i)=w_i$ and $g(v_i)=v_i+\sum_{k=0}^m s_{i+k}w_{k}$.
We need to compute $q_j(gv_i)$ for $j=0,1$ and $i=0, \ldots, m-1$.
We find
\begin{align*}
q_0(gv_i)&=q_0\left( v_i+\sum_{k=0}^{m} s_{i+k}w_{k} \right)\\
&= 
q_0(v_i)+b_0\left(v_i,\sum_{k=0}^{m} s_{i+k}w_{k}\right)
+q_0\left(\sum_{k=0}^{m} s_{i+k}w_{k}\right)\\
&= r_{2i+1} + s_{2i+1} + \sum_{k=0}^m s_{i+k}^2a_{2k},
\end{align*}
and similarly that
\[
q_1(gv_i) = r_{2i} + s_{2i} + \sum_{j=i}^{m+i} s_j^2a_{2j+1-2i}
\ .
\]
By Lemma~\ref{lem:BbasisSquare}, we conclude that the new invariant
is $r+s+s^2$ as desired.
\end{proof}

\section{Relations to the Arf invariant}
\label{sec:arf}

The following theorem was suggested to us by A.~Efimov.

Recall that any non-degenerate quadratic form in $2m$
variables over a field $\bbk$ of characteristic 2 can be reduced to the form
\[ q = \sum_{i=1}^m a_ix_i^2+x_iy_i+b_iy_i^2 \]
for some basis $x_1,\ldots,x_m,y_1,\ldots,y_m$.
The \emph{Arf invariant} of $q$ is 
\[ \Arf(q) = \sum_{i=1}^ma_ib_i\in \bbk/\wp(\bbk). \]
It is independent of a choice of a canonical form from above.

Let $(q_0,q_1)$ be a regular pair of quadratic forms
in normal form with invariants $\Delta$ and $r$.
We assume that $a_n \ne 0$ and define $A=\bbk[T]/(f(T))$
with $t$ the image of $T$ as in previous sections.
The pair gives rise to a quadratic form $q_A$ on the module
$E_A = E \otimes_\bbk A$ by defining
\[
q_A = q_0+tq_1
\]
as an element $S^2(E^\vee_A)$.
Since $A$ is a sum of fields $A_1,\ldots A_l$,
we may define the Arf invariant of $q_A$ to be
the sum of the Arf invariants of the restrictions $q_{A_i}$.

\begin{thm}
$q_A$ is an orthogonal sum of a $1$-dimensional trivial form
and a non-degenerate $2m$-dimensional form whose Arf invariant
is $r$ modulo $\wp(A)+\bbk$.
\end{thm}

\begin{proof}
We select a new basis for $E_A = E \otimes_\bbk A$ as follows:
\[
\begin{array}{rlr}
w_i' &= \sum_{k=i}^m w_i \otimes t^{k-i} & \textrm{ for $i$ in $0,\ldots,m$}\\
v_i' &= v_i \otimes 1 & \textrm{ for $i$ in $0,\ldots,m-1$}
\end{array}
\]
and find that
\[
b_A(w'_i,w'_j)=0\ , \quad
b_A(v'_i,v'_j)=0\ , \quad
b_A(v'_i,w'_j) = \delta_{(i+1)j},
\]
where $b_A$ is the bilinear form associated to $q_A$.
In particular, $w_0'$ is the image of $\Omega$ in $E_A$ and we see that
it spans a $1$-dimensional form where $q_A(w_0')=0$.

The remaining vectors span a subspace which decomposes into an
orthogonal sum of $2$-dimensional subspaces $\langle w'_{i+1}, v'_i
\rangle$ on which $q_A$ is isomorphic to the quadratic form
\[
q_A(w'_{i+1})x^2 + xy + q_A(v'_i)y^2
\]
for appropriate coordinates $x$ and $y$.

Note that
\[
q_A(w'_{i+1})=\sum_{k=i}^m (a_{2i}+a_{2i+1}t) t^{2(k-i)} = d_{2i+1}
\]
while
\[
q_A(v'_i)=r_{2i}t+r_{2i+1}.
\]

The Arf invariant of the $2m$-dimensional subspace is thus
\[
\sum_{i=0}^{m-1} q_A(w'_{i+1}) q_A(v'_i)
=\sum_{i=0}^{m-1} (r_{2i} t + r_{2i+1} ) d_{2i+1}
= \left( \sum_{i=0}^{n-2} r_id_i \right) + \left( \sum_{i=0}^{m-1}
r_{2i}a_{2i+1} \right)
\]
which is equal to $r \mod \bbk$.
\end{proof}

\section{Automorphisms of pencils of quadratic forms}
\label{sec:autos}

\subsection{Arithmetic description of automorphisms of a pair}

Let $A$ be the $\bbk$-algebra from the previous two sections.
Since the characteristic is $2$, the set of idempotents $\Idem(A)$ form
a subgroup of the additive group of $A$.
Recall that an idempotent $a$ in $A$ is an element such that $a^2=a$.
Consequently, the group $\Idem(A)$ can be characterized as the kernel of
the Artin-Schreier map
\[
\wp : A \to A
\]
which takes $a$ to $a^2+a$.

Let $\Aut(q_0,q_1)$ be the set of automorphisms of the pair $(q_0,q_1)$;
in other words, the subgroup of $\GL(E)$ whose elements induce an
isomorphism of the pair with itself.
From Theorem~\ref{thm:ArtinSchreierFormula}, we see that
every automorphism comes from an idempotent via \eqref{eq:gs}.
Conversely, the only non-trivial idempotent which gives rise to a trivial
automorphism is the multiplicative identity $1 \in \bbk \subset A$.
Thus we have the following:

\begin{thm} \label{thm:autIdem}
There is a canonical isomorphism
\[
\Aut(q_0,q_1) \cong \Idem(A)/\langle 1 \rangle
\]
where $\langle 1 \rangle$ is the additive subgroup of $A$ generated by
the unit element $1 \in A$.
\end{thm}

The group $\Idem(A)$ is generated by the orthogonal set of idempotents
$\epsilon_1, \ldots, \epsilon_l$ from \eqref{eq:epsilon}.
Since the additive subgroup of $A$ is commutative and every element has
order $2$,
we see that $\Idem(A)$ is an elementary abelian $2$-group of order
$2^l$.

Consequently, $\Aut(q_0,q_1)$ is an elementary abelian $2$-group of
order $2^{l-1}$ where $l$ is the number of irreducible factors in
$f(T)$.
When $A$ is a field (equivalent to $f(T)$ being irreducible),
then the only idempotent is $1$ and thus $\Aut(q_0,q_1)$ is trivial.
At the other extreme, when $f(T)$ splits completely,
$\Aut(q_0,q_1)$ has order $2^{2m}$ and we will see below that
the idempotents $\epsilon_1, \ldots, \epsilon_n$ correspond to
reflections.

$\Idem(A)$, respectively $\Aut(q_0,q_1)$, can be viewed as the set
of $\bbk$-points of a finite \'etale group scheme
$\underline{\Idem}(A)$, respectively $\underline{\Aut}(q_0,q_1)$,
over $\bbk$.
The group scheme $\underline{\Idem}(A)$ is simply the Weil restriction of
scalars $\mathbb{R}_{A/\bbk}( (\bbZ/2\bbZ)_A )$ where
$\bbZ/2\bbZ$ is the constant group scheme of order $2$.
The group scheme
$\underline{\Aut}(q_0,q_1)$
is the quotient in the exact sequence
\begin{equation} \label{eq:IdemAutSES}
\xymatrix{
0 \ar[r] & \bbZ/2\bbZ \ar[r]^-{\iota} &
\mathbb{R}_{A/\bbk}( (\bbZ/2\bbZ)_A ) \ar[r] &
\underline{\Aut}(q_0,q_1) \ar[r] & 0
}
\end{equation}
where $\iota$ is the natural monomorphism.

\begin{prop} \label{prop:GaloisCohomlogy}
There is a canonical isomorphism between the Galois cohomology group
$H^1(\bbk,\underline{\Aut}(q_0,q_1) )$
and the group $A/(\wp(A) + \bbk)$. 
\end{prop}

\begin{proof}
This follows the same reasoning as \S{}1~of~\cite{Skor}.
The composition
\[
H^2(\bbk,\bbZ/2\bbZ) \to
H^2(\bbk,\mathbb{R}_{A/\bbk}( (\bbZ/2\bbZ)_A )) =
H^2(A,\bbZ/2\bbZ)
\]
factors through the restriction-corestriction map,
which is multiplication by the odd number $\dim_\bbk(A)$.
Since $\bbZ/2\bbZ$ and $\mathbb{R}_{A/\bbk}( (\bbZ/2\bbZ)_A )$
have an even number of elements over a separable closure,
the cohomology groups are $2$-primary.
Thus the map $H^2(\bbk,\bbZ/2\bbZ) \to H^2(\bbk,\mathbb{R}_{A/\bbk}(
(\bbZ/2\bbZ)_A ))$ is injective.
Thus from \eqref{eq:IdemAutSES},
the group $H^1(\bbk,\underline{\Aut}(q_0,q_1)$ is the cokernel of the map
\[
H^1(\bbk,\bbZ/2\bbZ) \to H^1(\bbk,\mathbb{R}_{A/\bbk}( (\bbZ/2\bbZ)_A)) \ .
\]
The characteristic is $2$, thus $H^1(\bbk,\bbZ/2\bbZ) \cong \bbk/\wp(\bbk)$ and
\[
H^1(\bbk,\mathbb{R}_{A/\bbk}( (\bbZ/2\bbZ)_A)) = H^1(A,(\bbZ/2\bbZ)_A))
\cong A/\wp(A) \ ,
\]
so the result follows.
\end{proof}

\begin{remark}
The above proposition could be used to recover a weak form of
Theorem~\ref{thm:ArtinSchreierFormula}.
Of course, this argument would be circular in our presentation
since Theorem~\ref{thm:ArtinSchreierFormula} was used
to determine the automorphism group.
\end{remark}

\subsection{Geometric description of automorphisms of a pair}
\label{sec:geomAutos}

In this subsection we assume that $\bbk$ is algebraically closed.

Let $x_1, \ldots, x_n$ be vectors in $U$ representing the
points $\overline{x_1}, \ldots, \overline{x_n}$
in the zero locus of $V(\Delta)$ on $|U|$.
Recall that $V(\sfq_{x_1}), \ldots, V(\sfq_{x_n})$
are precisely the singular quadrics of the pencil $|\sfq|$.

Fix $i$ in $1, \ldots, n$.
Since $V(\sfq_{x_i})$ is of corank $1$, it contains a unique
singular point $\overline{z_i} \in |E|$ represented by the vector
$z_i = \Omega(x_i)$ in $E$.

For any $u \in U$ such that $|u| \ne |x_i|$ in $|U|$ consider the
following linear automorphism of $E$
\begin{equation} \label{eq:def_rho_i}
\rho_i(v) = v + \frac{\sfb_u(z_i,v)}{\sfq_u(z_i)}z_i
\end{equation}
where $v \in E$.
A priori, the formula for $\rho_i$ depends on the choice of $u \in U$.
However, one can check that $\rho_i$ does not depend on the choice of
$u$ as long as it is not a multiple of $x_i$
(use that $\sfq_{x_i}(z_i)=0$ and $\sfb_{x_i}(z_i,v)=0$
for any choice of $v \in E$).

\begin{thm}
The automorphism group $\Aut(q_0,q_1)$ is an elementary
abelian $2$-group of order $2^{2m}$
generated by the reflections
$\rho_1, \ldots, \rho_n$
subject to the relation
$\rho_1 \cdots \rho_n=1$.
\end{thm}

\begin{proof}
It is clear that the reflections $\{\rho_i\}$ are automorphisms, that
they are of order $2$, and that they commute.
It remains to show that there are no other automorphisms and to show
that they satisfy the desired relations and no others.
We will show that the reflections $\rho_1, \ldots, \rho_n$ correspond to
a set of orthogonal idempotents of $A$.
The result then follows by Theorem~\ref{thm:autIdem}.

The remainder of the proof is simply a direct calculation.
In view of Remark~\ref{rem:technicalNuisance},
we may assume without loss of generality that $q_0$ is nondegenerate
and each vector $x_i$ has coordinates $(\alpha_i,1)$ where
$\alpha_i$ is a root of $f(T)$.
Let $\epsilon_i$ be the idempotent corresponding to $\alpha_i$;
in other words, multiplication by $\epsilon_i$ gives rise to a map
$A \to \bbk$ such that $t \mapsto \alpha_i$ as in \eqref{eq:idemAction}.

We will establish that $\phi(\epsilon_i)=\rho_i$
where $\phi : A \to \GL(E)$ is the map \eqref{eq:phiDef}.
To do this, we determine the values of $s_0, \ldots, s_{n-2}$
in \eqref{eq:gs}.
From \ref{eq:sInvDef},
these are the first $n-1$ coordinates of $\epsilon_i$
in the basis $d_0,\ldots,d_{n-1}$ of $A$ defined in
\eqref{eq:def_d_i}.
Thus, by Proposition~\ref{prop:dualBasis}, we have
\[
s_j = \tr_{A/\bbk}\left(\epsilon_i \frac{t^i}{f'(t)} \right)
= \frac{\alpha_i^j}{f'(\alpha_i)}
\]
for each $j$ in $0, \ldots, n-2$.

Now, note that $z_i = \sum_{i=0}^m \alpha^i w_i$.
Since $f'(T)=\sum_{k=0}^m a_{2i+1} T^{2i}$,
we have $f'(\alpha_i) = q_1(z_i)$.
We have $b_1(v_j,z_i)=\alpha_i^j$ for $j$ in $0, \ldots, m-1$
where $w_0,\ldots,w_m,v_0,\ldots,v_{m-1}$ is the Kronecker basis
for $E$.

Putting these observations together,
we see that
\[
\phi(\epsilon_i)(v_j) - v_j = \sum_{k=0}^m s_{j+k}w_k
= \frac{\alpha_i^j}{f'(\alpha_i)} \sum_{k=0}^m \alpha^k w_k
= \frac{b_1(z_i,v_j)}{q_1(z_i)}z_i
\]
for all $j$ in $0, \ldots, m-1$.
Since, additionally,
$\phi(\epsilon_i)(w_j)=\rho(w_j)=w_j$ for $j$ in $0, \ldots, m$,
we conclude that that $\phi(\epsilon_i)(v)=\rho(v)$
for all $v \in E$.
\end{proof}

With this description of the automorphism group, we have the following:

\begin{corollary} \label{cor:generators}
There are exactly $2^{2m}$ generators permuted simply transitively by
the group $\Aut(q_0,q_1)$.
\end{corollary}

\begin{proof}
By Corollary~\ref{cor:generatorIsL}, the generators are in bijective
correspondence with subspaces $L$ occurring in a Kronecker normal form
such that $L$ is totally isotropic with respect to every quadratic form
in the pencil.  The group $\Aut(q_0,q_1)$ permutes such spaces.
The subspace $L$ is never invariant under a non-trivial
automorphism as in \eqref{eq:gs}, thus the action is simply transitive.
The number of generators is $2^{2m}$ since that is the order of the
group.
\end{proof}

From the above corollary, we prove Theorem~\ref{thm:GaloisCohomology}
from the introduction:

\begin{proof}[Proof of Theorem~\ref{thm:GaloisCohomology}]
The set of generators of $X$ corresponds to a smooth $0$-dimensional
subscheme of an appropriate Grassmanian.
In view of Corollary~\ref{cor:generators},
over a non-closed field this is an 
$\underline{\Aut}(q_0,q_1)$-torsor.
From Proposition~\ref{prop:GaloisCohomlogy}
we have that
$A/(\wp(A) + \bbk) \cong H^1(\bbk,\underline{\Aut}(q_0,q_1) )$.
Thus the $r$-invariant in $A$ naturally gives rise to a class in
$H^1(\bbk,\underline{\Aut}(q_0,q_1)$.
\end{proof}

\subsection{Automorphisms of the base locus}

The goal of this section is to prove Theorem~\ref{thm:autos}
which describes the automorphisms of the base scheme 
$\Bs(\sfq) = V(q_0,q_1)$ in $\bbP^{n-1}$. 

\begin{lem} \label{lem:Fano}
Two smooth complete intersections of pairs of quadrics in $\bbP^{2m}$ are
isomorphic if and only if there is an element of $\PGL_{2m+1}$
inducing the isomorphism.
\end{lem}

\begin{proof}
If $m=1$ then the two varieties are smooth subschemes of $\bbP^2$,
of dimension $0$ of degree $4$; the result is clear in this case.
Otherwise, when $m \ge 2$,
the adjunction formula gives that the canonical sheaf $\omega_X$ is
isomorphic to $\calO_X(-n+4)$.
Since $n =2m+1 \ge 5$, the anticanonical divisor 
class $-K_X$ is a positive multiple of the hyperplane section.
Hence the embedding of $X$ in
$\bbP^{2m}$ comes from a multiple of the anticanonical sheaf.
Any isomorphism between the varieties induces isomorphisms of the
corresponding anticanonical sheaves, and the result follows.
\end{proof}

Theorem~\ref{thm:autos} follows immediately from the following.

\begin{thm} \label{thm:autosPlus}
Assume $X$ is quasi-split. 
The automorphism group $\Aut(X)$ of $X$ is isomorphic to
$\Aut(q_0,q_1) \rtimes G$
where $G$ is the subgroup of $\PGL(U)$ which leaves
invariant the scheme $V(\Delta)$ of zeros of the half-discriminant.
\end{thm}

\begin{proof}
By Lemma~\ref{lem:Fano}, any automorphism of $X$ can be represented
by an element of $\PGL(E)$.

Recall that $|W|$ is a canonical subspace and must be invariant under any
automorphism of $X$.
From Proposition~\ref{prop:coordFree},
$W$ is canonically isomorphic to $S^m(U^\vee)^\vee$.
Thus an automorphism $h \in \PGL(E)$ fixes $|W|$ pointwise
if and only if it fixes $|U|$ pointwise.
Thus there is a group homomorphism $\pi : \Aut(X) \to \PGL(U)$ whose kernel is
precisely the group $\Aut(q_1,q_2)$ of automorphisms of the pair $(q_1,q_2)$.

Note that since the zeroes of $\Delta$ correspond to the singular quadrics 
of the pencil, all automorphisms must leave invariant the zeroes of
$\Delta$.  Thus the image of $\pi$ is contained in $G$.  It remains to
show that any element $g \in \GL(U)$ representing an element of $G$
can be lifted to $\Aut(X)$.

Since $X$ is quasi-split, we have a normal form for the pair
$(q_{u_1},q_{u_2})$ where $r=0$.
We may scale our representative $g \in \GL(U)$ such that
$g^*\Delta=\Delta$.
By Lemma~\ref{lem:changePair}, there is an $h \in \GL(W) \times \GL(L)$
so that $(\sfq_{gu_0},\sfq_{gu_1})$ has a normal form
with respect to the basis $h(v_1), \ldots, h(w_{m})$.
Since $g^*\Delta=\Delta$,
we have $\sfq_{g(u_i)}(w_j)=\sfq_{u_i}(w_j)$
for $i=0,1$ and $j=0, \ldots, m$.
Since $L$ is totally isotropic,
$\sfq_{g(u_i)}(v_j)=\sfq_{u_i}(v_j)=0$
for $i=0,1$ and $j=0, \ldots, m-1$.
Thus the normal forms are equal and $g$ is an automorphism as desired.
\end{proof}

\begin{remark}
Note in the last step of the proof, we required that $r=0$
in order to lift
automorphisms from $g \in \PGL(U)$ to $\Aut(X)$.
Thus, when $X$ is not quasi-split, the automorphism group may
not surject onto $G$.
\end{remark}

\section{Cohomology of intersection of two quadrics in
\texorpdfstring{$\bbP^{2m}$}{P2m}}
\label{sec:cohomology}

Let $X$ be a smooth intersection of two quadrics in $\bbP^{2m}$ over an
algebraically closed field of characteristic $p\ge 0$. Recall that in
the case $m = 2$ when $X$ is a del Pezzo surface of degree 4, the Picard
group $\Pic(X)$ of algebraic 2-cycles on $X$ is a free abelian group of
rank 6 and the cycle homomorphism $\Pic(X)_{\bbZ_\ell}\to H^2(X,\bbZ_\ell)$
to the $\ell$-adic cohomology is an isomorphism. It also compatible with
the intersection  pairing $\Pic(X)\times \Pic(X)\to \bbZ$ and the
cup-product $H^2(X,\bbZ_\ell)\times H^2(X,\bbZ_\ell)\to \bbZ_\ell$ on
the cohomology. This well-known result follows from the fact that $X$ is
isomorphic to the blow-up of 5 points $p_1,\ldots,p_5$ in the projective
plane $\bbP^2$ no three of which are collinear. In the anti-canonical
model, the exceptional curves of the blow-up are 5 disjoint lines on
$X$. The remaining eleven lines come from the proper transforms of lines
$\overline{p_i,p_j}$ and the conic passing through the five points.

Let $e_1,\ldots,e_5$ be the classes of the exceptional curves, and $e_0$
be the class of the pre-image of a line in the plane under the blow-up.
Then $\Pic(X)$ is freely generated  by $e_0,\ldots,e_5$ and the
canonical class of $X$ is equal to $-3e_0+e_1+\cdots+e_5$. Its
orthogonal complement in $\Pic(X)$ is isomorphic to the negative root
lattice of type $D_5$.
The group of automorphisms of $X$ is faithfully
represented in the Weyl group of this lattice
and is a subgroup to $2^4\rtimes \frakS_5$,
where $\frakS_5$ denotes the symmetric group on $5$ letters.

In this section we extend these well known facts to the case of
arbitrary $m$
(for $\cha(\Bbbk) \ne 2$ see \cite{Reid}).
Throughout, $X$ will be the base locus of a regular pencil $\sfq$ of
quadrics and $\sfP$ will denote the projectivized radical subspace
$|W|$ of $|E|$.

Fix a generator $\Lambda = |L|$ in $X$.
Let $X_\Lambda\to X$ be the blow-up of $X$ with center at $\Lambda$.
Consider the projection map $p_\Lambda:X_\Lambda\to \bbP^{m}$ from
$\Lambda$.
We have $E = W \oplus L$ so we may take $p_\Lambda$ to be the
projectivization of the projection $E\to W$ and thus identify
$\bbP^{m}$ with $\sfP = |W|$.
For any point $x\in X\setminus \Lambda$,
$p_\Lambda(x) = \la \Lambda,x\ra\cap |W|$.
The base locus of the restriction of the pencil
$\sfq$ to the subspace $\la \Lambda,x\ra\cong \bbP^m$
contains $\Lambda$.
Thus the residual pencil consists of hyperplanes in $\la \Lambda,x\ra$
containing $x$.
Its base locus is a codimension 2 linear subspace of $\la \Lambda,x\ra$
containing $x$ unless one of the quadrics contains $\la \Lambda,x\ra$
and hence the two quadrics intersect along a hyperplane $\Lambda'$ in
$\la \Lambda,x\ra$.
If $\Lambda'$ exists, then it is a generator in $X$ intersecting
$\Lambda$ along a hyperplane.
Note that $\Lambda' = \rho_i(\Lambda)$ for
some reflection $\rho_i\in \Aut(q_0,q_1)$ and $\Lambda'\cap \Lambda$ is the
intersection of the fixed hyperplane $F_i$ of $\rho_i$ with $\Lambda'$.

The image of each such generator  under the projection map $p_\Lambda$
is the point $y_i = [v_i]$ in $|W|$ which is the singular point  of one
of the singular
quadrics in the pencil. In particular, we found $2m+1$ generators that
do not intersect each other outside $\Lambda$ but each intersects
$\Lambda$ along a hyperplane. Their proper transforms on the blow-up are
disjoint subvarieties each isomorphic to $\bbP^{m-1}$. This gives the
following birational picture of $X$.

Let $\bbP^{2m}_\Lambda\to \bbP^{2m}$ be the blow-up of $\bbP^{2m}$ along
$\Lambda$. The projection from $\Lambda$ defines an isomorphism from
$\bbP^{2m}_\Lambda$ to the projective bundle $\bbP(\calE)$, where
$\calE = \calO_{\bbP^m}^{\oplus m}\oplus \calO_{\bbP^m}(-1)$.

\begin{prop}
$X_\Lambda$ is isomorphic over $\sfP \cong \bbP^m$ to a closed
subvariety of the projective $(m+1)$-bundle $\bbP(\calE)$ over $\sfP$.
Let $\Sigma = \{y_1,\ldots,y_{2m+1}\}$ be the images in $|W|$ of the
vertices of singular quadrics in the pencil.
If $x\not\in \Sigma$ (resp. $x\in \Sigma$), then
the fiber $p_\Lambda^{-1}(x)$ is a codimension 2 (resp. codimension  1)
linear subspace in the fiber of the projective bundle.
\end{prop} 

The inclusion $X_\Lambda\hookrightarrow \bbP(\calE)$ is
given by a surjective homomorphism of coherent sheaves $\calE\to \calF$
such that $X_\Lambda \cong \bbP(\calF):=\Proj~S^\bullet(\calF)$.  The
set $\Sigma$ is equal to the singular set $\Sing(\calF)$, the set of
points $x\in \bbP^m$ such that $\dim_\bbk(\calF(x)) > \rank(\calF)$.
In our case $\rank(\calF) = m-1$ and $\dim_\bbk\calF(x) = m$
for $x \in \Sing(\calF)$.

Let $U =
\bbP^m\setminus \Sigma$ and $j:U\hookrightarrow \bbP^m$ be the open
inclusion. Since $\textrm{codim}~\Sigma \ge 2$, the sheaf
$j_*j^*\calF$ is a locally free sheaf of rank $m-1$ and $\calF$ is
its subsheaf; in particular, it is a reflexive sheaf \cite{Hartshorne}.
Thus, we obtain

\begin{prop}\label{eq:reflexive} 
\[ X_\Lambda \cong \bbP(\calF), \]
where $\calF$ is a reflexive sheaf of rank $m-1$ with $\Sing(\calF) =
\Sigma.$ In particular, $X$ is birationally equivalent to a projective
$(m-2)$-bundle over $\bbP^m$.
\end{prop} 
 
\begin{example}
Assume $m = 2$, then $p_\Lambda$ is the projection of the quartic del
Pezzo surface $X$ from a line. It is isomorphic to the blow-up
$\bbP(\calF)$, where $\calF$ is the ideal sheaf $\calI_\Sigma$ of the
set $\Sigma$ of five points in the plane.
\end{example}

\begin{remark} \label{rem:ptInPlaneModel}
Recall that, over an algebraically closed field, a del Pezzo surface of
degree $4$ is obtained by blowing up $5$ general points in the plane
$\bbP^2$ which lie on a unique conic $C = V(q)$, where $q$ is a
nondegenerate quadratic form.
Yu. Prokhorov asked us for a geometric interpretation in the plane model
of the canonical point $\Pi$ from Theorem~\ref{thm:canPlane}.

A natural guess is that it is the \emph{strange point} of the conic.
It has the property that any line through this point is tangent to $C$.
This turns out to be false.
In fact, our conic depends
only on the pencil of the associated alternate bilinear forms. In a
Kronecker basis it can be given by equation $x_1x_3+x_2^2 = 0$. The
equation of $\Pi$ depends on the coefficients $(a_0,\ldots,a_5)$ of the
half-discriminant, and coincides with the strange point only in the case
when $a_2a_5 =a_3a_4, a_0a_3 = a_1a_2.$
\end{remark}

The following is well known when the characteristic is not $2$;
see Chapter~3~of~\cite{Reid} for a proof over $\bbC$.

\begin{prop}
Let $A_{m-1}(X)$ be the Chow group of algebraic $(m-1)$-cycles on $X$
equipped with a structure of a quadratic lattice with respect to the
interesection of algebraic cycles on $X$. Then 
\[ A_{m-1}(X) \cong \bbZ^{2m+2} \ . \]
Let $A_{m-1}(X)_0$ denote the orthogonal complement of $K_X^{m-1}$ in
$A_{m-1}(X)$. Then it is isomorphic to the roots lattice of type
$D_{2m+1}$ taken with the sign $(-1)^{m-1}$.
\end{prop}

\begin{proof}
Let $U = \bbP^m\setminus \Sigma$, let $Y_\Sigma = p_\Lambda^{-1}(\Sigma)$,
and let $X_U = p_\Lambda^{-1}(U)$.
By Proposition~1.8~from~\cite{Fulton}, we have the following exact
sequence of Chow groups $A_k$ of algebraic $k$-cycles:

\[\xymatrix{
A_{k}(Y_{\Sigma})\ar[r]&A_{k}(X_\Lambda)\ar[r]&A_{k}(X_U)\ar[r]&0.}\]

For a vector bundle $\calE$ of rank $e+1$ over a base $Z$,
the Chow groups of the projective bundle $\bbP(\calE)$ are well known.
By Theorem~3.3~of~\cite{Fulton}, we have
\[
A_k(\bbP(\calE)) = \bigoplus_{i=0}^e A_{k-e+i}(Z) \ .
\]
It follows that 
\[
A_{m-1}(X_U) \cong \bbZ^{m-1}, \ A_{m-1}(Y_\Sigma) \cong \bbZ^{2m+1} \ .
\]
The homomorphism $A_{m-1}(Y_\Sigma)\to A_{m-1}(X_\Sigma)$ is
injective because $Y_\Sigma$ is the union of disjoint subvarieties of
dimension $m-1$.
This shows that  $A_{m-1}(X_\Lambda) \cong \bbZ^{3m}$.
It remains to use \cite{Fulton}, Proposition 6.7 (e) that computes the
cohomology of the blow-up.
We have $A_{m-1}(X_\Lambda) = A_{m-1}(X) \oplus \bbZ^{m-2}$.
This gives $A_{m-1}(X) \cong \bbZ^{2m+2}$.

The rest of the arguments can be borrowed from Chapter~3~of~\cite{Reid}
since they do not depend on the characteristic of the ground field.
We include this for completeness sake.
First, we fix one generator $\Lambda$ and use
Corollary \ref{cor:generators} to index generators by subsets of the set
$[1,2m+1] = \{1,2,\ldots,2m+1\}$ modulo taking the complementary subset.
Here $\Lambda_\emptyset := \Lambda$ is the generator corresponding to
the empty set,
and each $\Lambda_{\{i\}} := \rho_i(\Lambda)$ is the generator
corresponding to the image of $\Lambda$ under the reflection $\rho_i$.

We equip $A_{m-1}(X)$ with a structure of a nondegenerate quadratic
lattice with respect to the intersection pairing $A_{m-1}(X)\times
A_{m-1}(X)\to \bbZ$. We have
\[ [\Lambda_I]\cdot [\Lambda_J] =
(-1)^r\left(\left\lfloor\frac{r}{2}\right\rfloor+1\right),\]
 where $r = \dim \Lambda_I\cap \Lambda_J.$ In
particular,
\begin{equation}\label{eq:miles1}
[\Lambda_I]^2 = (-1)^{m-1}
\left(\left\lfloor\frac{m-1}{2}\right\rfloor+1\right)
\end{equation}
and, if $\dim \Lambda_I\cap \Lambda_J = m-3$,
\begin{equation} \label{eq:miles2}
\left([\Lambda_I]-[\Lambda_J]\right)^2 = 2(-1)^{m-1}.
\end{equation}
Let $\eta$ be the class of a hyperplane section of $X$. We have
\[\eta^{m-1}\cdot [\Lambda_I] = 1.\]
 Since $K_X = (3-2m)\eta$,  we have
$K_X^{m-1} = (3-2m)^{m-1}\eta^{m-1}$, hence
\[A_{m-1}(X)_0:=(K_X^{m-1})^\perp = (\eta^{m-1})^\perp.\]
 We may assume that the
generators $\Lambda_i, i = 1,\ldots,2m+1$, are the pre-images of the
points $y_i\in \Sigma$ under the projection map $p_\Lambda:X_\Lambda \to
|W|$.  Let $e_i = [\Lambda_i]$ be their classes in $A_{m-1}(X)$. Let
\[e_0 = \eta^{m-1}-[\Lambda_\emptyset] \]
(note that in the case $m = 2$ it
coincides with the class of the pre-image of a line in $\bbP^2$).  One
checks that the classes $e_0,e_1,\ldots,e_{2m+1}$ freely
generate $A_{m-1}(X)$ and the classes
\[
\alpha_0 = -e_0+[\Lambda_\emptyset]+e_{2m}+e_{2m+1}, \
\alpha_i = e_i-e_{i+1},\ \
i = 1,\ldots 2m,
\]
freely generate $A_{m-1}(X)_0$. In fact they form
a root basis of type $D_{2m+1}$ in $A_{m-1}(X)_0$ (taken with the sign
$(-1)^{m-1}$).
\end{proof}

Recall that the Weyl group $W(D_n)$ of the root lattice of type
$D_{n}$ generated by reflections in simple roots $\alpha_i$ is
isomorphic to $2^{n-1}\rtimes \frakS_{n}$. The group $\Aut(X)$ acts
naturally on $A_{m-1}(X)$ leaving $K_X$ invariant. This defines a
homomorphism 
\[ \rho:\Aut(X) \to \Or(A_{m-1}(X)_0) \cong
\Or(\sfD_{n}).
\]
to the orthogonal group of the quadratic lattice
$\sfD_{n}$ defined by the Cartan matrix of the root system of type
$D_{n}$. The image is contained in the subgroup of $\Or(\sfD_{n})$
of isometries that can be lifted to $A_{m-1}(X)$. It follows from the
theory of quadratic lattices that this subgroup is of index 2 and
coincides with the Weyl group $W(D_{n})$. Thus, we have defined a
homomorphism 
\[ \rho:\Aut(X) \to W(D_{n}) \cong 2^{n-1}\rtimes \frakS_{n}. \]

\begin{thm}
The homomorphism $\rho$ is injective and sends the subgroup
$\Aut(q_0,q_1)$ to the subgroup $2^{n-1}$ of $W(D_{n})$.
\end{thm}

\begin{proof}
An element $g$ of the kernel fixes all generators. Assume $p\ne 2$. We
know that any generator $\Lambda$ intersects $n=2m+1$ other generators
that intersect $\Lambda$ along a hyperplane. We know from the
description of the projection map $p_\Lambda:X_\Lambda\to |W|$ given in
the beginning of this section that $\Lambda$ intersects $n$
generators $\Lambda_{i}$ that are projected to the singular points
$y_i\in |W|$ of $n$ singular fibers from the pencil. Thus $g$ in its
action in $|W|$ fixes these points, and since they generate $|W|$ fixes
$|W|$ pointwise. Thus the projection of $g$ to $G$ from Theorem
\ref{thm:autosPlus} is the identity, hence $g$ is the identity. 
\end{proof}

\section*{Acknowledgments}
The second author would like to thank D.~Leep for helpful conversations
and A.~Efimov for suggesting the interpretation via the Arf invariant
used in Section~\ref{sec:arf}.
Both authors thank A.~Skorobogatov and an anonymous referee for comments.


\begin{thebibliography}{20}

\bibitem[B90]{Bhosle} U.~Bhosle.
\newblock Pencils of quadrics and hyperelliptic curves in characteristic
two.
\newblock {\em J. Reine Angew. Math.} 407 (1990), 75--98.


\bibitem[B03]{Bhosle1} U.~Bhosle.
\newblock Moduli of vector bundles in characteristic 2.
\newblock {\em Math. Nachr}. 254/255 (2003), 11Ð-26.


\bibitem[B70]{Bourbaki} N.~Bourbaki.
\newblock \'El\'ements de math\'ematique. Alg\`ebre. Chapitres 1 \`a 3.
\newblock {Hermann, Paris 1970} 


\bibitem[DM98]{Debarre} O.~Debarre, L.~Manivel.
\newblock Sur la vari\'et\'e des espaces lin\'eaires contenus dans une
intersection compl\`ete.
\newblock {\em Math. Ann.} 312 (1998), no.
3, 549--574.


\bibitem[DR77]{Desale} U.~Desale (Bhosle) and S.~Ramanan.
\newblock Classification of vector bundles of rank 2 on hyperelliptic
curves.
\newblock {\em Invent. Math.} 38 (1976/77), 549--574.

\bibitem[D12]{CAG}
I.~Dolgachev.
\newblock {\em Classical algebraic geometry: a modern view}.
\newblock Cambridge University Press, Cambridge, 2012.

\bibitem[BE77]{Buchsbaum} D.~Buchsbaum, D. Eisenbud.
\newblock{\em  Algebra structures for finite free resolutions, and some
structure theorems for ideals of codimension 3}.
\newblock {\em Amer.  J. Math.} 99 (1977), no. 3, 447--485.

\bibitem[EKM08]{EKM} R. Elman, N. Karpenko, and A. Merkurjev.
\newblock {\em The algebraic and geometric theory of quadratic forms},
volume~56 of {\em American Mathematical Society Colloquium
Publications}.
\newblock American Mathematical Society, Providence, RI, 2008.

\bibitem[F98]{Fulton} W.~Fulton.
\newblock {\em Intersection theory. Second edition}, Ergebnisse der
Mathematik und ihrer Grenzgebiete. 3. Folge. A Series of Modern Surveys
in Mathematicsf {\em Series of Modern Surveys in Mathematics}.
\newblock Springer-Verlag, Berlin, 1998.

\bibitem[G55]{Gau} L.~Gauthier.
\newblock {\em Footnote to a footnote of Andr\'e Weil.}
\newblock {Univ.  Politec. Torino. Rend. Sem. Mat.} 56 (1954--55), 325--328.

\bibitem[EGA5]{EGA} A.~Grothendieck.
\newblock EGA 5 (translation and edition of Grothenedieck's prenotes by
P. Blass and J.  Blass).
\newblock \url{http://www.jmilne.org/math/Documents/EGA-V.pdf}.

\bibitem[H80]{Hartshorne} R.~Hartshorne.
\newblock Stable reflexive sheaves.
\newblock {\em Math. Ann.} no. 2, 254 (1980), 121--176.

\bibitem[II15]{Ish} Y.~Ishitsuka and T. Ito.
\newblock Sheaf theoretic classifications of pairs of square matrices
over arbitrary fields.
\newblock arXiv:1503.07603 [math.AG], 2015.

\bibitem[K08]{Kneser} M. Kneser.
\newblock {\em Quadratische Formen.  Revised and edited in collaboration
with Rudolf Scharlau}, volume~56 of {\em ASpringer-Verlag, Berlin,
2002.}
\newblock American Mathematical Society, Providence, RI, 2008.

\bibitem[L94]{Lang94} S. Lang.
\newblock {\em Algebraic number theory}, volume 110 of {\em Graduate
Texts in Mathematics}.
\newblock Springer-Verlag, New York, second edition, 1994.

\bibitem[LS99]{LS99}
D. Leep and L. Schueller.
\newblock Classification of pairs of symmetric and alternate bilinear forms.
\newblock {\em Exposition. Math.}, 17(5):385--414, 1999.

\bibitem[LS02]{LS02} D. Leep and L. Schueller.
\newblock A characterization of nonsingular pairs of quadratic forms.
\newblock {\em J. Algebra Appl.}, 1(4):391--412, 2002.

\bibitem[MT86]{MT} Yu.~Manin and M.~Tsfasman.
\newblock Rational varieties: algebra, geometry, arithmetic.
\newblock {\em Uspekhi Mat.  Nauk}, 41 (1986), no. 2(248), 43--94.

\bibitem[M99]{Muth} P.~Muth.
\newblock Theorie und Anwendung der Elementarteiler.
\newblock Leipzig, Teubner, 1899.

\bibitem[R72]{Reid} M.~Reid.
\newblock The complete intersection of two or more quadrics.
\newblock Ph. D. Thesis, Trinity College, Cambridge. 1972.  

\bibitem[S10]{Skor} A.~Skorobogatov.
\newblock del Pezzo surfaces of degree 4 and their relation to Kummer
surfaces.
\newblock {\em Enseign. Math.}, (2) 56 (2010), no. 1-2, 73--85.

\bibitem[V05]{Beauville} C. Voll.
\newblock
Functional equations for local normal zeta functions of nilpotent groups.
With an appendix by A. Beauville. 
\newblock {\em Geom. Funct. Anal.} 15 (2005), no. 1, 274--295.

\bibitem[W13]{Wang} X.~Wang.
\newblock
Maximal linear spaces contained in the base loci of pencils of quadrics.
\newblock arXiv1302.2385 [math.AG], 2013.

\bibitem[W76]{Wat76}
W. Waterhouse.
\newblock Pairs of quadratic forms.
\newblock {\em Invent. Math.}, 37(2):157--164, 1976.

\bibitem[W77]{Wat77}
W. Waterhouse.
\newblock Pairs of symmetric bilinear forms in characteristic {$2$}.
\newblock {\em Pacific J. Math.}, 69(1):275--283, 1977.

\bibitem[W68]{Weir}
K.~Weierstrass.
\newblock Zur Theorie bilineare und quadratischen Formen.
\newblock {\em Monatsberichte der K\"oniglich Preussischen Akademie der Wiss.
zu Berlin}, 1868, 310--338.

\end{thebibliography}
\end{document}